\documentclass[a4paper,10pt]{article}
\usepackage[T1]{fontenc}
\usepackage{amssymb,amsmath,amsthm,mathtools,bbold, tikz,rotating,cite,todonotes, hyperref}
\usepackage[all,cmtip]{xy}
\usepackage[enableskew]{youngtab}
\usepackage[left=2.5cm,top=2.5cm,bottom=2.5cm,right=2.5cm]{geometry}
\usepackage[british]{babel}

\newtheorem{lemma}{Lemma}[section]
\newtheorem{theorem}[lemma]{Theorem}
\newtheorem{corollary}[lemma]{Corollary}
\newtheorem{proposition}[lemma]{Proposition}

\newtheorem{conjecture}[lemma]{Conjecture}

\theoremstyle{plain}
\newtheorem{defi}[lemma]{Definition}
\theoremstyle{definition}
\newtheorem{example}[lemma]{Example}
\newtheorem{remark}[lemma]{Remark}
\newtheorem{proofpart}{Part}

\DeclareMathOperator{\cat}{Cat}
\DeclareMathOperator{\id}{Id}

\DeclareMathOperator{\supp}{Supp}
\DeclareMathOperator{\SYT}{SYT}

\DeclareMathOperator{\sh}{sh}

\newcommand{\R}{\mathbb{R}}
\newcommand{\C}{\mathbb{C}}

\newcommand{\wt}{wt}
\newcommand{\ten}{$10$}
\newcommand{\eleven}{$11$}
\newcommand{\twelve}{$12$}
\newcommand{\thirteen}{$13$}
\newcommand{\fourteen}{$14$}
\newcommand{\fifteen}{$15$}
\newcommand{\sixteen}{$16$}
\newcommand{\seventeen}{$17$}
\newcommand{\eighteen}{$18$}
\newcommand{\nineteen}{$19$}
\newcommand{\twenty}{$20$}
\newcommand{\smallo}{o}
\newcommand\numberthis{\addtocounter{equation}{1}\tag{\theequation}}

\title{Partial sum of matrix entries of representations of the symmetric group and its asymptotics}
\author{Dario De Stavola \thanks{email: dario.destavola@math.uzh.ch\newline Institute of Mathematics, University of Zurich, Winterthurerstrasse 190, 8057, Zurich, Switzerland.\newline Keywords: Jucys-Murphy elements, symmetric functions, symmetric group.}}

\begin{document}
\maketitle

\begin{abstract}
Many aspects of the asymptotics of Plancherel distributed partitions have been studied in the past fifty years, in particular the limit shape, the distribution of the longest rows, connections with random matrix theory and characters of the representation matrices of the symmetric group. Regarding the latter, we expand a celebrated result of Kerov on the asymptotic of Plancherel distributed characters by studying partial trace and partial sum of a representation matrix. We decompose these objects into a main term and a reminder, proving a central limit theorem for both main terms and a law of large numbers for the partial sum itself. Our main tool is the expansion of symmetric functions evaluated on Jucys-Murphy elements. 
\end{abstract}

\section{Introduction}\label{sec:in}
Let $\lambda$ be a partition of $n$, in short $\lambda\vdash n$, represented as a Young diagram in English notation. A filling of the boxes of $\lambda$ with numbers from $1$ to $n$, increasing towards the right and downwards, is called a \emph{standard Young tableau}. We call $\dim\lambda$ the number of standard Young tableaux of shape $\lambda$. We fix $n$ and we associate to each $\lambda$ the probability $\frac{(\dim\lambda)^2}{n!}$, which defines the \emph{Plancherel measure}.
\smallskip

Let us recall briefly three results for the study of the asymptotics of Plancherel distributed random partitions. They relate algebraic combinatorics, representation theory of the symmetric group, combinatorics of permutations, and random matrix theory. 
\begin{enumerate}
 \item The partitions of $n$ index the irreducible representations of the symmetric group $S_n$. For each $\lambda\vdash n$ the dimension of the corresponding irreducible representation is $\dim\lambda$. A natural question concerns the asymptotics of the associated characters when $\lambda$ is distributed with the Plancherel measure. A central limit theorem was proved with different techniques by Kerov, \cite{kerov1993gaussian}, \cite{ivanov2002kerov}, and Hora, \cite{hora1998central}:
 \begin{equation}\label{intro2}
  n^{\frac{|\rho|-m_1(\rho)}{2}}\hat{\chi}^{\lambda}_{\rho}\overset{d}\to \prod_{k\geq 2} k^{m_k(\rho)/2} \mathcal{H}_{m_k(\rho)}(\xi_k).
 \end{equation}
 Here $\rho$ is a partition of $n$, $\hat{\chi}^{\lambda}_{\rho}$ is the renormalized character associated to $\lambda$ calculated on a permutation of cycle type $\rho$, $m_k(\rho)$ is the number of parts of $\rho$ which are equal to $k$, $\mathcal{H}_m(x)$ is the $m$-th Hermite polynomial, and $\{\xi_k\}_{k\geq 2}$ are i.i.d. standard gaussian variables.
 
\item The Robinson-Schensted-Knuth algorithm allows us to interpret the longest increasing subsequence of a uniform random permutation as the first row of a Plancherel distributed partition. This motivates the study of the shape of a random partition. A limit shape result was proved independently by \cite{kerov1977asymptotics} and \cite{logan1977variational}, then extended to a central limit theorem by Kerov, \cite{ivanov2002kerov}. For an extensive introduction on the topic, see \cite{romik2015surprising}.

\item More recently it was proved that, after rescaling, the limiting distribution of the longest $k$ rows of a Plancherel distributed partition $\lambda$ coincides with the limit distribution of the properly rescaled $k$ largest eigenvalues of a random Hermitian matrix taken from the Gaussian Unitary Ensemble. See for example \cite{borodin2000asymptotics} and references therein. Such similarities also occur for fluctuations of linear statistics, see \cite{ivanov2002kerov}.
\end{enumerate}
In the aftermath of Kerov's result \eqref{intro2}, a natural step in the study of the characters of the symmetric group is to look at the representation matrix rather than just the trace. We consider thus, for a real valued matrix $A$ of dimension $N$ and $u\in [0,1]$, the partial trace and partial sum defined, respectively, as

\[PT_u(A):=\sum_{i\leq uN}\frac{A_{i,i}}{N},\qquad PS_u(A):= \sum_{i,j\leq uN}\frac{A_{i,j}}{N}.\]
We study these values when $A=\rho^{\lambda}(\sigma)$ is a representation matrix of $S_n$, where $\lambda$ is a partition of $n$ and $\sigma$ is a permutation of $S_r$, $r\leq n$. We are interested in the asymptotic of $PT_u(\rho^{\lambda}(\sigma))$ and $PS_u(\rho^{\lambda}(\sigma))$ when $n$ grows. These partial sums are obviously not invariant by isomorphisms of representations, hence we consider an explicit natural construction of irreducible representations (the Young seminormal representation). 
We define \emph{subpartitions} $\mu_j$ of $\lambda$, denoted $\mu_j\nearrow\lambda$, as partitions of $n-1$ obtained from $\lambda$ by removing one box. The seminormal representation, recalled in Section \ref{section young seminormal}, allows a decomposition of the partial trace and partial sum of a representation matrix: we will show that there exist $\bar{j}$ and $\bar{u}$ such that
\begin{equation}\label{intro1}
 PT_u^{\lambda}(\sigma):=PT_u(\pi^{\lambda}(\sigma))=\sum_{j<\bar{j}}\frac{\dim\mu_j}{\dim\lambda}PT_1^{\mu_j}(\sigma)+\frac{\dim\mu_{\bar{j}}}{\dim\lambda}PT_{\bar{u}}^{\mu_{\bar{j}}}(\sigma),
\end{equation}
\begin{equation}\label{intro3}
 PS_u^{\lambda}(\sigma):=PS_u(\pi^{\lambda}(\sigma))=\sum_{j<\bar{j}}\frac{\dim\mu_j}{\dim\lambda}PS_1^{\mu_j}(\sigma)+\frac{\dim\mu_{\bar{j}}}{\dim\lambda}PS_{\bar{u}}^{\mu_{\bar{j}}}(\sigma).
\end{equation}
Here $PT_1^{\mu_j}(\sigma)=\hat{\chi}^{\mu_j}(\sigma)$, while 
\[PS^{\lambda}_1(\sigma)=\sum_{i\leq \dim\lambda}\frac{\pi^{\lambda}(\sigma)_{i,i}}{\dim\lambda}=:TS^{\lambda}(\sigma)\]
is the \emph{total sum} of the matrix $\rho^{\lambda}(\sigma)$.

In the first (resp. the second) decomposition we call the first term the \emph{main term for the partial trace} $MT_u^{\lambda}(\sigma)$ (resp. \emph{main term for the partial sum} $MS_u^{\lambda}(\sigma)$) and the second the \emph{remainder for the partial trace} $RT_u^{\lambda}(\sigma)$ (resp. the \emph{remainder for the partial sum} $RS_u^{\lambda}(\sigma)$). The decompositions show that the behavior of partial trace and partial sum depend on, respectively, total trace and total sum. 

As recalled in Equation \eqref{intro2}, a central limit theorem for total traces, or characters, is well known; we prove a central limit theorem for the total sum (Theorem \ref{convergence of total sum}): to each $\sigma \in S_r$ we associate the two values 
\[m_{\sigma}:= \mathbb{E}_{PL}^{r}\left[TS^{\nu}(\sigma)\right]\qquad \mbox{and}\qquad v_{\sigma}:=\binom{r}{2}\mathbb{E}_{PL}^{r}\left[\hat{\chi}^{\nu}_{(2,1,\ldots,1)} TS^{\nu}(\sigma)\right],\]
where $\mathbb{E}_{PL}^r[X]$ is the average of the random variable $X$ considered with the Plancherel measure $(\dim\nu)^2/r!$ for $\nu\vdash r$.
\begin{theorem}
Fix $\sigma\in S_{r}$ and let $\lambda\vdash n$. Then
 \[\sqrt{n}\cdot(TS^{\lambda}(\sigma)-m_{\sigma})\overset{d}\to \mathcal{N}(0,2 v_{\sigma}^2),\]
 where $\mathcal{N}(0,2 v_{\sigma}^2)$ is a normal random variable of variance $2 v_{\sigma}^2$.
\end{theorem}
The idea is to show that, for $\sigma\in S_r$, the total sum $TS^{\lambda}(\sigma)$ can be written as linear combination of $\{\hat{\chi}^{\lambda}(\tau)\}_{\tau\in S_r}$.
\bigskip

When investigating the partial trace, we focus on the main term, and we prove the following theorem:
\begin{theorem}\label{ora}
 let $\sigma$ be a permutation of cycle type $\rho$ and $u\in [0,1]$. Let $\{\xi_k\}_{k\geq 2}$ be a sequence of independent standard Gaussian variables. Then 
 \[n^{\frac{|\supp(\sigma)|}{2}}MT_{u}^{\lambda}(\sigma)\overset{d}\to u\cdot \prod_{k\geq 2} k^{m_k(\rho)/2} \mathcal{H}_{m_k(\rho)}(\xi_k),\]
 where $|\supp(\sigma)|$ is the size of the support of the permutation $\sigma$.
\end{theorem}
In short, this theorem states that if we condition the partial trace of a representation matrix by the total trace, the partial trace appears to be deterministic. We actually prove multivariate version of the theorems stated in this introduction. In particular, for Theorem \ref{ora}, the joint distributions of $n^{\frac{|\supp(\sigma_i)|}{2}}MT_{u_i}^{\lambda}(\sigma_i)$ will converge to Hermite polynomials of the same Gaussian variables for a family $\{\sigma_i\}$ of permutations and $\{u_i\}$ of real numbers. Notice that this result generalizes \eqref{intro2}, although we use Kerov's result in the proof.

Informally, the main idea to prove Theorem \ref{ora} is to show that when $n$ grows,
\[\sum_{j<\bar{j}}\frac{\dim\mu_j}{\dim\lambda}\hat{\chi}^{\mu_j}(\sigma)\sim \left(\sum_{j<\bar{j}}\frac{\dim\mu_j}{\dim\lambda}\right)\hat{\chi}^{\lambda}(\sigma).\]
\smallskip

To achieve this result we need to estimate the asymptotic of $\hat{\chi}^{\lambda}(\sigma)-\hat{\chi}^{\mu}(\sigma)$, for $\lambda$ distributed with the Plancherel measure and $\mu\nearrow\lambda$. We consider power sums $p_{\nu}$, where $\nu$ is the cycle type of $\sigma$ in which we remove $1$ to each part, calculated on the multiset of contents $\mathcal{C}_{\lambda}$. We prove in Section \ref{Asymptotic of the main term} that
\begin{equation}
 p_{\nu}(\mathcal{C}_{\lambda})-p_{\nu}(\mathcal{C}_{\mu})\in o_P(n^{-\frac{|\supp(\sigma)|}{2}}) \qquad \mbox{implies}\qquad \hat{\chi}^{\lambda}(\sigma)-\hat{\chi}^{\mu}(\sigma)\in o_P(n^{-\frac{|\supp(\sigma)|}{2}}).
\end{equation}
 This will be accomplished by introducing modified power sums $\tilde{p}_{\nu}$ and by analyzing its highest terms with the appropriate filtration. We consider then $ p_{\nu}(\mathcal{C}_{\lambda})-p_{\nu}(\mathcal{C}_{\mu})$, which can be estimated by an expansion of the power sums.
\medskip

We cannot unfortunately prove asymptotic results on the remainder $RT_u^{\lambda}(\sigma)$, although we conjecture that 
\begin{equation}\label{conjecture on remainder}
n^{\frac{|\rho|-m_1(\rho)}{2}}RT_u^{\lambda}(\sigma)\overset{p}\to 0 
\end{equation}
for $u\in [0,1]$. In Section \ref{section conjecture} we describe a different conjecture, which would imply the one above, involving quotient of dimensions of irreducible representations. We give some numerical evidence. Our conjecture would imply
\[n^{\frac{|\rho|-m_1(\rho)}{2}}PT_u^{\lambda}(\sigma)\overset{d}\to u\cdot \prod_{k\geq 2} k^{m_k(\rho)/2} \mathcal{H}_{m_k(\rho)}(\xi_k).\]
\bigskip

Regarding the partial sum, our results on the total sum and the main term of the partial sum imply a law of large numbers for the partial sum (Theorem \ref{convergence of partial sum}):
\begin{theorem}
 Let $u\in \R$, $\sigma\in S_r$ and $\lambda\vdash n$. Then 
 \[ PS_u^{\lambda}(\sigma)\overset{p}\to u\cdot m_{\sigma}.\]
\end{theorem}
It is easy to see that the reminder for the partial sum goes asymptotically to zero, but we do not know how fast. For the same reasons as above, we cannot thus present a central limit theorem for the partial sum. Nevertheless, we prove a central limit theorem for the main term of the partial sum (Corollary \ref{convergence of G}): 
\begin{theorem}
  Let $u\in \R$, $\sigma\in S_r$ and $\lambda\vdash n$. Then  
 \[\sqrt{n}\cdot\left(MS_u^{\lambda}(\sigma)-u\cdot m_{\sigma}\right)\overset{d}\to u\cdot\mathcal{N}(0,2 v_{\sigma}^2).\]
\end{theorem}
As for the partial trace case, we show a multivariate generalization of the previous two theorems.
\medskip

It is worth mentioning that the partial trace $PT_u(A)$ and the partial sum $PS_u(A)$ have been studied by D'Aristotile, Diaconis and Newman in \cite{d2003brownian} for the case in which $A$ is a random matrix of the Gaussian Unitary Ensemble (GUE). The authors showed that in this case both partial trace and partial sum, after normalization, converge to Brownian motion, and thus it has a higher degree of randomness than the partial sum and partial trace of a representation matrix.

\medskip
In section 2 we recall results on the co-transition measure, the Young seminormal representation, partial permutations, shifted symmetric functions and we introduce the partial trace. In section 3 we prove the result concerning the asymptotics of the main term. In section 4 we study the total and partial sum, while in section 5 we describe a conjecture which would imply a convergence result on the partial trace.

\section{Preliminaries}\label{section Preliminaries}

\subsection{Notation}

Set $\lambda=(\lambda_1,\lambda_2,\ldots)\vdash n$ to be a partition of $n$, {\it i.e.} a nonincreasing sequence of positive integers whose sum is $n$. We associate partitions with \emph{Young diagrams} represented in English notation, as in the example below. In this setting a \emph{box} (in symbol $\Box$) is an element $(a,b)\in\mathbb{N}\times \mathbb{N}$, and we write $\Box=(a,b)\in\lambda$ if $1\leq a\leq \lambda_b$. We say that a box $\Box\in\lambda$ is a \emph{outer corner} if there exists another Young diagram with the same shape as $\lambda$ without that box. Likewise, a box $\Box\notin\lambda$ is a \emph{inner corner} if there exists a Young diagram with the same shape of $\lambda$ together with $\Box$. For a box $\Box=(a,b)$ the \emph{content} is defined as $c(\Box)=c(a,b)=a-b$. Of particular interest are the contents of outer and inner corners of $\lambda$, and we call them respectively $y_j$ and $x_j$, ordered in a way such that $x_1<y_1<x_2<\ldots<y_d<x_{d+1}$, where $d$ is the number of outer corners. For a outer corner box $\Box$ such that $y_j=c(\Box)$ for some $j$, the partition of $n-1$ resulting from removing that box is called $\mu_j$, and we write $\mu_j\nearrow\lambda$. We call such a $\mu_j$ a \emph{subpartition} of $\lambda$. Similarly, $\Lambda_j$ indicates the partition of $n+1$ obtained from adding the inner corner box of content $x_j$. 
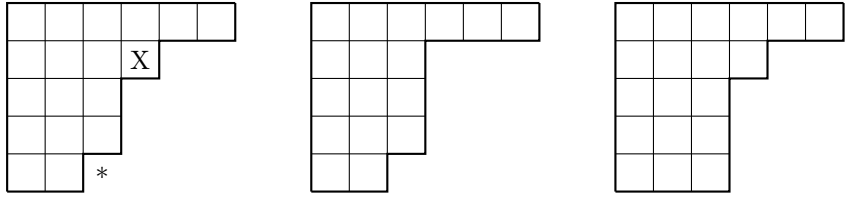
\begin{figure}

 \begin{center}
 \begin{tikzpicture}[scale=0.5]
\draw[thick](0,0)--(2,0)--(2,1)--(3,1)--(3,3)--(4,3)--(4,4)--(6,4)--(6,5);
\draw[thick](0,0)--(0,5)--(6,5);
\draw[thick](8,0)--(10,0)--(10,1)--(11,1)--(11,4)--(12,4)--(14,4)--(14,5);
\draw[thick](8,0)--(8,5)--(14,5);
\draw[thick](16,0)--(19,0)--(19,3)--(20,3)--(20,4)--(22,4)--(22,5);
\draw[thick](16,0)--(16,5)--(22,5);
\draw (1,0)--(1,5);
\draw (2,1)--(2,5);
\draw (3,3)--(3,5);
\draw (4,4)--(4,5);
\draw (5,4)--(5,5);
\draw (9,0)--(9,5);
\draw (10,1)--(10,5);
\draw (11,4)--(11,5);
\draw (12,4)--(12,5);
\draw (13,4)--(13,5);
\draw (17,0)--(17,5);
\draw (18,0)--(18,5);
\draw (19,3)--(19,5);
\draw (20,4)--(20,5);
\draw (21,4)--(21,5);
\draw (0,1)--(2,1);
\draw (0,2)--(3,2);
\draw (0,3)--(3,3);
\draw (0,4)--(4,4);
\draw (8,1)--(10,1);
\draw (8,2)--(11,2);
\draw (8,3)--(11,3);
\draw (8,4)--(11,4);
\draw (16,1)--(19,1);
\draw (16,2)--(19,2);
\draw (16,3)--(19,3);
\draw (16,4)--(20,4);
\node at (2.5,0.5){$\ast$};
\node at (3.5,3.5){X};
 \end{tikzpicture}
 \end{center}
 \caption{Example of Young diagrams. Starting from the left: $\lambda=(6,4,3,3,2)$, $\mu_3=(6,3,3,3,2)$ and $\Lambda_2=(6,4,3,3,3)$. }\label{example young diagrams}
 \end{figure}

 \begin{example}
 In Figure \ref{example young diagrams} we show three Young diagrams. The first is a partition $\lambda=(6,4,3,3,2)\vdash 18$ where we stress out a inner corner of content $-2$ and a outer corner of content $2$. The second is a subpartition corresponding to removing the box $\begin{array}{c}   \young(X) \end{array}$ from $\lambda$, while the third has the same shape of $\lambda$ with an additional box corresponding to $\begin{array}{c}   \young(\ast) \end{array}$.
\end{example}
We study the irreducible representation $\pi^{\lambda}(\sigma)$ associated to $\lambda$, and we call $\dim\lambda$ its dimension. It is a known fact that the set
\[\SYT(\lambda):=\{T: T \mbox{ is a standard Young tableau of shape }\lambda\}\]
indexes a basis for $\pi^{\lambda}$, so that $\dim\lambda=|\SYT(\lambda)|$.
\begin{example}\label{example on tableaux}
We present the set of standard Young tableaux of shape $\lambda=(3,2)$.
 \[ T_1=\begin{array}{c}\young(123,45)\end{array}\quad T_2=\begin{array}{c}\young(124,35)\end{array}\quad T_3=\begin{array}{c}\young(134,25)\end{array}\quad T_4=\begin{array}{c}\young(125,34)\end{array}\quad T_5=\begin{array}{c}\young(135,24)\end{array}.\]
\end{example}
  By setting $P_{PL}(\lambda)=(\dim\lambda)^2/n!$ we endow the set of partitions of $n$ with a probability measure, called the \emph{Plancherel measure} of size $n$. Throughout the paper we always assume that $\lambda$ is distributed with the Plancherel measure.
\subsection{Transition and co-transition measures}
Fix $\lambda\vdash n$. Two identities follow from the branching rule (see for example\hbox{\cite[section 2.8]{sagan2013symmetric}}):
\[\sum\limits_{\mu\nearrow\lambda}\dim\mu=\dim\lambda,\qquad \sum\limits_{\lambda\nearrow\Lambda}\dim\Lambda=(n+1)\dim\lambda.\]
 For $v\in\mathbb{R}$ we can define, respectively, the \emph{normalized co-transition distribution} and the \emph{normalized transition distribution}:
\[F_{ct}^{\lambda}(v):=\sum\limits_{y_j\leq v\sqrt{n} }\frac{\dim\mu_j}{\dim\lambda},\quad F_{tr}^{\lambda}(v):=\sum\limits_{x_j\leq v\sqrt{n} }\frac{\dim\Lambda_j}{(n+1)\dim\lambda}.\]
In \cite{kerov1993transition} Kerov proved that the normalized transition distribution converges almost surely to the semicircular distribution, when $\lambda$ is equipped with the Plancherel measure. Building on Kerov's work, we prove the same result for the normalized co-transition distribution.
\begin{lemma}\label{convergence of co-transition}
 When the set of partitions of $n$ is equipped with the Plancherel measure, then almost surely
 \[\lim_{n \to +\infty}F_{ct}^{\lambda}(v)=\frac{1}{2\pi}\int_{-2}^v \sqrt{4-t^2}\,dt\]
 holds for all $|v|\leq 2.$
\end{lemma}
\begin{proof}
We prove the lemma by showing that the Stieltjes transform of the normalized co-transition measure associated to the random variable $\lambda^{(n)}$ converges to the Stieltjes transform of the semicircular distribution. It is showed in \cite[corollary 1]{geronimo2003necessary} that pointwise convergence of Stieltjes transforms implies convergence in distribution, which again implies convergence of distribution functions at continuity points (recall that the semicircular distribution function is continuous everywhere). We claim thus that
\begin{equation}\label{Stieltjes transform equation}
 ST^{\lambda^{(n)}}_{ct}(u):=\sum_{\mu_j\nearrow\lambda}\frac{\dim\mu_j}{\dim\lambda^{(n)}}\frac{1}{u-y_j}\overset{a.s.}\to \frac{u-\sqrt{u^2-4}}{2}=:ST_{sc}(u).
\end{equation}
In \cite[sections 2 and 4]{kerov1993transition}, Kerov showed that the Stieltjes transform of the transition measure converges almost surely to the Stieltjes transform of the semicircular distribution. Moreover, he also showed that, if $Q^{\lambda}(u):=\prod_j (u-y_j)$ and $P^{\lambda}(u):=\prod_j (u-x_j)$, then 
\[ST^{\lambda^{(n)}}_{tr}(u)=\frac{Q^{\lambda}(u)}{P^{\lambda}(u)},\qquad ST^{\lambda^{(n)}}_{ct}(u)=u-\frac{P^{\lambda}(u)}{Q^{\lambda}(u)}.\]
Hence
\[ST^{\lambda^{(n)}}_{ct}(u)=u-\frac{1}{ST^{\lambda^{(n)}}_{tr}(u)}\overset{a.s.}\to u-\frac{2}{u-\sqrt{u^2-4}}=\frac{u-\sqrt{u^2-4}}{2}.\]
This proves \eqref{Stieltjes transform equation} and hence the lemma.
\end{proof}

\begin{figure}
\begin{center}
 \begin{tikzpicture}[scale=0.5]
\draw[->] (-1,0) -- (12,0);
\node at (12,0)[below right]{$v$};
\draw[->] (0,-1) -- (0,8);
\node at (0,8)[left]{$\tilde{u}$};
\draw[line width=.5mm](0,0)--(2,0);
\draw[fill=black] (0,0) circle (.07cm);
\draw (2,0) circle (.07cm);
\draw[line width=.5mm](2,2)--(3,2);
\draw[fill=black] (2,2) circle (.07cm);
\draw (3,2) circle (.07cm);
\draw[line width=.5mm](3,3)--(5,3);
\draw[fill=black] (3,3) circle (.07cm);
\draw (5,3) circle (.07cm);
\draw[line width=.5mm](5,6)--(8,6);
\draw[fill=black] (5,6) circle (.07cm);
\draw (8,6) circle (.07cm);
\draw[line width=.5mm](8,7)--(9,7);
\draw[fill=black] (8,7) circle (.07cm);
\draw (9,7) circle (.07cm);
 \node at (-0.3,2.5)[left] {$u$};
 \draw[dashed] (-0.3,2.5)--(12,2.5);
 \draw[dashed] (3,3)--(3,-0.3);
 \node at (3,-0.3)[below]{$\left(F_{ct}^{\lambda}\right)^{*}(u)=v^{\lambda}$};
  \draw[dashed] (8,7)--(8,-0.3);
   \node at (8,-0.3)[below]{$\frac{y_j}{\sqrt{n}}$};
 \end{tikzpicture}
\end{center}
 \caption{Example of graph of $\tilde{u}=F_{ct}^{\lambda}(v)$.}\label{picture co transition measure}
 \end{figure}
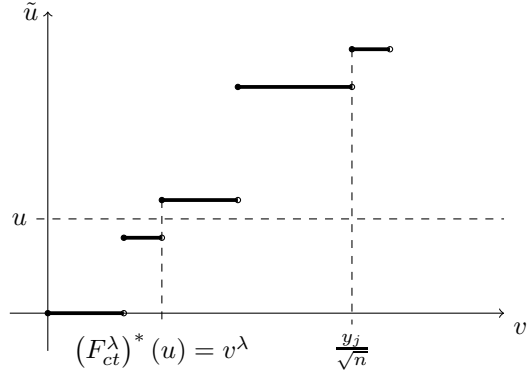

Set 
\[\left(F_{ct}^{\lambda}\right)^{*}(u):=\sup\left\{z\in\mathbb{R} \mbox{ s.t. } F_{ct}^{\lambda}(z)\leq u\right\}.\]
We consider $\left(F_{ct}^{\lambda}\right)^{*}$ as the inverse of the step function $F_{ct}^{\lambda}$. We want to show that $F_{ct}^{\lambda}(v^{\lambda})$ converges to $u$:
\begin{lemma}\label{behaviour of co-transition}
 For $v^{\lambda}=\left(F_{ct}^{\lambda}\right)^{*}(u)$ for a fixed $u$, then
 \[F_{ct}^{\lambda}(v^{\lambda})=\sum_{y_j\leq v^{\lambda}\sqrt{n}}\frac{\dim\mu_j}{\dim\lambda}\overset{p}\to u.\]
\end{lemma}
\begin{proof}
We show in Figure \ref{picture co transition measure} an example of a normalized co-transition distribution. Since $F_{ct}^{\lambda}$ is right continuous then $u\leq F_{ct}^{\lambda}(v^{\lambda})$. Moreover, since the co-transition distribution is a step function, for each element of the image $\tilde{u}\in F_{ct}^{\lambda}(\mathbb{R})$ there exists $j$ such that $F_{ct}^{\lambda}(\frac{y_j}{\sqrt{n}})=\tilde{u}$. Call $\bar{j}$ the index corresponding to $F_{ct}^{\lambda}(v^{\lambda})$, that is, $F_{ct}^{\lambda}(\frac{y_{\bar{j}}}{\sqrt{n}})=F_{ct}^{\lambda}(v^{\lambda})$. Hence $\sum_{j<\bar{j}}\dim\mu_j/\dim\lambda\leq u$. Thus
\begin{itemize}
 \item $v^{\lambda}\leq \frac{y_{\bar{j}}}{\sqrt{n}}$, since $F_{ct}^{\lambda}(\frac{y_{\bar{j}}}{\sqrt{n}})\geq u$;
 \item $v^{\lambda}\geq \frac{y_{\bar{j}}}{\sqrt{n}}$, since for each $\epsilon>0$, $F_{ct}^{\lambda}(\frac{y_{\bar{j}}}{\sqrt{n}}-\epsilon)\leq u$.
\end{itemize}
Thus $v^{\lambda}= \frac{y_{\bar{j}}}{\sqrt{n}}$. Therefore
\[\sum_{j<\bar{j}}\frac{\dim\mu_j}{\dim\lambda} =F_{ct}^{\lambda}(v^{\lambda})-\frac{\dim\mu_{\bar{j}}}{\dim\lambda}\leq u\leq F_{ct}^{\lambda}(v^{\lambda}).\]
 Equivalently
 \[-\frac{\dim\mu_{\bar{j}}}{\dim\lambda}\leq u- F_{ct}^{\lambda}(v^{\lambda})\leq 0,\]
 and $\max \frac{\dim\mu_{j}}{\dim\lambda}\overset{p}\to 0$ because of the convergence of the normalized co-transition distribution towards an atom free distribution proved in the previous lemma. This proves the statement.
 \end{proof}

\subsection{Young seminormal representation and last letter order}\label{section young seminormal}

We recall the definition of Young seminormal representation, see \cite{young1977collected} for the original introduction, and \cite{GreeneRationalIdentity} for a more modern description. We need three preliminary definitions; set $\lambda\vdash n$, then  
\begin{enumerate}
 \item given a box $\Box\in \lambda$ recall that the \emph{content} $c(\Box)$ is the difference between the row index and the column index of the box. For $k\leq n$ we also write $c_k(T)$ for the content of the box containing the number $k$ in the tableau $T$. As instance, in the tableau $T_1$ of Example \ref{example on tableaux}, $c_4(T_1)=-1$, $c_1(T_1)=c_5(T_1)=0$, $c_2(T_1)=1$ and $c_3(T_1)=2$.
 \item For $k\leq n-1$ the \emph{signed distance} between $k$ and $k+1$ in the tableau $T$ is $d_k(T):= c_k(T)-c_{k+1}(T)$. For example $d_1(T_1)=d_2(T_1)=d_4(T_1)=-1$; $d_3(T_1)=3$.
 \item If $k$ and $k+1$ are in different columns and rows we define $(k,k+1)T$ as the standard Young tableau equal to $T$ but with the boxes containing $k$ and $k+1$ inverted. In the previous example $(3,4)T_1=T_2$, while $(2,3)T_1$ is not defined.
\end{enumerate}
\begin{defi}\label{def young seminormal}
 The Young seminormal representation defines the matrix associated to \hbox{$ \pi^{\lambda}((k,k+1)) $} entrywise in the following way: if $T$ and $S$ are standard Young tableaux of shape $\lambda$, then
\[\pi^{\lambda}((k,k+1))_{T,S}=\left\{\begin{array}{cr}
1/d_k(T)&\mbox{ if } T=S;\\
\sqrt{1-\frac{1}{d_k(T)^2}}&\mbox{ if } (k,k+1)T=S;\\
0&\mbox{else.}\end{array}\right.
\]
\end{defi}

Notice that the adjacent transpositions $(k,k+1)$, $k=1,\ldots, n-1$, generate the group $S_n$, hence $\pi^{\lambda}(\sigma)$ is well defined for all $\sigma\in S_n$. 
\begin{example}\label{example of matrices}
 Consider $\lambda=(3,2)$, $\sigma=(2,4,3)$. We compute $\pi^{\lambda}((2,4,3)):$
  \[\pi^{\lambda}((2,4,3))=\pi^{\lambda}((3,4)(2,3))=\pi^{\lambda}((3,4))\pi^{\lambda}((2,3))\]
   \[=\left[\begin{array}{ccccc}
    -1/3&{\scriptstyle\sqrt{8/9}}&0&0&0\\
    {\scriptstyle\sqrt{8/9}}&1/3&0&0&0\\
    0&0&1&0&0\\
    0&0&0&1&0\\
    0&0&0&0&-1
   \end{array}\right]\cdot
  \left[\begin{array}{ccccc}
    1&0&0&0&0\\
    0&-1/2&{\scriptstyle\sqrt{3/4}}&0&0\\
    0&{\scriptstyle\sqrt{3/4}}&1/2&0&0\\
    0&0&0&-1/2&{\scriptstyle\sqrt{3/4}}\\
    0&0&0&{\scriptstyle\sqrt{3/4}}&1/2
   \end{array}\right]\]
 \[=  \left[\begin{array}{ccccc}
    -1/3&-{\scriptstyle\sqrt{2/9}}&{\scriptstyle\sqrt{2/3}}&0&0\\
    {\scriptstyle\sqrt{8/9}}&-1/6&{\scriptstyle\sqrt{1/12}}&0&0\\
    0&{\scriptstyle\sqrt{3/4}}&1/2&0&0\\
    0&0&0&-1/2&{\scriptstyle\sqrt{3/4}}\\
    0&0&0&-{\scriptstyle\sqrt{3/4}}&-1/2
   \end{array}\right]\]
\end{example}

We define also the \emph{last letter order} in $\SYT(\lambda)$: let $T,S$ be standard Young tableaux, then $T\leq S$ if the box containing $n$ lies in a row in $T$ which is lower that $S$; if $n$ lies in the same row in both tableaux, then we look at the rows containing $n-1$, and so on. Notice that in the list of tableaux of shape $(3,2)$ of Example \ref{example on tableaux} the tableaux are last letter ordered; also, the entries of the matrices in Example \ref{example of matrices} are ordered accordingly. We write $\pi^{\lambda}(\sigma)_{i,j}$ instead of $\pi^{\lambda}(\sigma)_{T_i,T_j}$, where $T_i$ is the $i-$th tableau of shape $\lambda$ in the last letter order.

\subsection{The partial trace and its main term}\label{chapter the partial trace}
Let $\lambda\vdash n$  and $\sigma\in S_r$ a permutation of the set $\{1,\ldots,r\}$ with $r\leq n$. We see $S_r$ as a subgroup of $S_n$ by adding fixed points, so that $\pi^{\lambda}(\sigma)$ is well defined. Similarly, if $\rho\vdash r\leq n$, then we define $\chi^{\lambda}_{\rho}=\chi^{\lambda}_{\rho\cup 1^{n-r}}$ and $\hat{\chi}^{\lambda}_{\rho}=\hat{\chi}^{\lambda}_{\rho\cup 1^{n-r}}$.
\begin{defi}
 Let $\lambda\vdash n$, $\sigma\in S_r$, $u\in [0,1]$. Define the partial trace as
 \[PT_{u}^{\lambda}(\sigma):=\sum_{i\leq u \dim\lambda}\frac{\pi^{\lambda}(\sigma)_{i,i}}{\dim\lambda}.\]

\end{defi}
When $u=1$ then the partial trace correspond to the normalized trace 
\[\hat{\chi}^{\lambda}(\sigma)=\frac{\chi^{\lambda}(\sigma)}{\dim\lambda}=\sum_{i\leq \dim\lambda}\frac{\pi^{\lambda}(\sigma)_{i,i}}{\dim\lambda}.\]
Our choice of order on the tableaux implies that, as long as $\sigma\in S_r$ and $r\leq n-1$,
\begin{equation}\label{decomposition of the representation matrix}
 \pi^{\lambda}(\sigma)= \left[
\begin{array}{c c c c}
\pi^{\mu_1}(\sigma) &\mathbb{0}&\mathbb{0}&\cdots\\ 
\mathbb{0} &\pi^{\mu_2}(\sigma) &\mathbb{0}\\
\mathbb{0}&\mathbb{0}&\pi^{\mu_3}(\sigma) \\
\vdots&&&\ddots
\end{array}\right].
\end{equation}

Hence $\hat{\chi}^{\lambda}(\sigma)=\sum_j \frac{\dim\mu_j}{\dim\lambda}\hat{\chi}^{\mu_j}(\sigma)$. This decomposition of the total trace can be easily generalized to the partial trace:
\begin{proposition}[Decomposition of the partial trace]\label{first order decomposition of the partial trace}
 Fix $u\in [0,1]$, $\lambda\vdash n$ and $\sigma\in S_r$ with $r\leq n-1$. Set $\left(F_{ct}^{\lambda}\right)^{*}(u)=\frac{y_{\bar{j}}}{\sqrt{n}}=v^{\lambda}$ for some $\bar{j}$ as in Lemma \ref{behaviour of co-transition}. Define
 \[\bar{u}=\frac{\dim\lambda}{\dim\mu_{\bar{j}}}\left(u-\sum_{j<\bar{j}}\dim\mu_j\right)<1, \]
 then
\begin{equation}\label{first decomposition of partial trace}
 PT_u^{\lambda}(\sigma)=\sum_{j<\bar{j}}\frac{\dim\mu_j}{\dim\lambda}\hat{\chi}^{\mu_j}(\sigma)+\frac{\dim\mu_{\bar{j}}}{\dim\lambda}PT_{\bar{u}}^{\mu_{\bar{j}}}(\sigma).
\end{equation}
\end{proposition}
\begin{proof}
 We can decompose the partial trace as
 \[PT_u^{\lambda}(\sigma)=\sum_{i\leq u\dim\lambda}\frac{\pi^{\lambda}(\sigma)_{i,i}}{\dim\lambda}=\sum_{i\leq \sum\limits_{j<\bar{j}}\dim\mu_j}\frac{\pi^{\lambda}(\sigma)_{i,i}}{\dim\lambda}+\sum_{\sum\limits_{j<\bar{j}}\dim\mu_j<i\leq u\dim\lambda}\frac{\pi^{\lambda}(\sigma)_{i,i}}{\dim\lambda}.\]
 The first sum in the RHS is 
 \[\sum_{i\leq \sum\limits_{j<\bar{j}}\dim\mu_j}\frac{\pi^{\lambda}(\sigma)_{i,i}}{\dim\lambda}=\sum_{j<\bar{j}}\frac{\dim\mu_j}{\dim\lambda}\hat{\chi}^{\mu_j}(\sigma).\]
 We consider now the second sum of the RHS: we have $u\leq \sum\limits_{j\leq\bar{j}}\frac{\dim\mu_j}{\dim\lambda}$ by the definition of $\bar{j}$, thus
 \[\sum\limits_{j<\bar{j}}\dim\mu_j<i\leq u\dim\lambda\leq \sum\limits_{j\leq\bar{j}}\dim\mu_j.\]
Hence $\pi^{\lambda}(\sigma)_{i,i}=\pi^{\mu_{\bar{j}}}(\sigma)_{\tilde{i},\tilde{i}}$, where $\tilde{i}=i-\sum\limits_{j<\bar{j}}\dim\mu_j$, so that
\[0<\tilde{i}\leq u\dim\lambda-\sum\limits_{j<\bar{j}}\dim\mu_j=\bar{u}\dim\mu_{\bar{j}}.\]
Therefore 
\[\sum_{\sum\limits_{j<\bar{j}}\dim\mu_j<i\leq u\dim\lambda}\frac{\pi^{\lambda}(\sigma)_{i,i}}{\dim\lambda}=\sum_{\tilde{i}\leq \bar{u}\dim\mu_{\bar{j}}}\frac{\pi^{\lambda}(\sigma)_{\tilde{i},\tilde{i}}}{\dim\lambda}=\frac{\dim\mu_{\bar{j}}}{\dim\lambda}PT_{\bar{u}}^{\mu_{\bar{j}}}(\sigma),\]
and the proposition is proved.
\end{proof}

We define $\sum_{y_j\leq v^{\lambda}\sqrt{n}}\frac{\dim\mu_j}{\dim\lambda}\hat{\chi}^{\mu_j}(\sigma)$ to be the \emph{main term for the partial trace} (denoted $MT^{\lambda}_u(\sigma)$), while $\frac{\dim\mu_{\bar{j}}}{\dim\lambda}PT_{\bar{u}}^{\mu_{\bar{j}}}(\sigma)$ is called the \emph{remainder} (denoted $R^{\lambda}_u(\sigma)$). The main goal of this paper is to establish the asymptotic behavior of the main term for the partial trace. We settle some notation: if a permutation $\sigma$ has cycle type $\rho$ we write $m_k(\sigma)=m_k(\rho)$ for the number of parts of $\rho$ which are equal to $k$ or, equivalently, for the number of cycles of $\sigma$ which are of length $k$. We define the \emph{weights} of $\sigma$ and $\rho$ as $\wt(\sigma):=\wt(\rho):=|\supp(\sigma)|=|\rho|-m_1(\rho).$ We recall now a result due to \cite{kerov1993gaussian}, with a complete proof given by \cite{ivanov2002kerov} and, independently, by \cite{hora1998central}:
\begin{theorem}\label{convergence of characters}
Fix $\rho_1\vdash r_1, \rho_2\vdash r_2,\ldots$. The asymptotic behavior of the irreducible character $\hat{\chi}^{\lambda}$ is given by:
\[\left\{n^{\frac{\wt(\rho_i)}{2}}\hat{\chi}^{\lambda}_{\rho_i}\right\}\overset{d}\to \left\{\prod_{k\geq 2} k^{m_k(\rho_i)/2} \mathcal{H}_{m_k(\rho_i)}(\xi_k)\right\},\]
where $\{\xi_k\}_{k\geq 2}$ is a sequence of independent standard Gaussian random variables, and $\mathcal{H}_m(x)$, $m\geq 1$, is the Hermite polynomial of degree $m$ defined by the recurrence relation $x\mathcal{H}_m(x)=\mathcal{H}_{m+1}(x)+m\mathcal{H}_{m-1}(x)$ and initial data $\mathcal{H}_0(x)=1$ and $\mathcal{H}_1(x)=x$.
\end{theorem}
Our result is the following:
\begin{theorem}\label{convergence of transformed co-transition}
 let $\sigma_1,\sigma_2,\ldots$ be permutations of cycle type respectively $\rho_1,\rho_2,\ldots$ and $u_1,u_2,\ldots\in [0,1]$. Let $\{\xi_k\}_{k\geq 2}$ be a sequence of independent standard Gaussian variables. Then 
 \[\left\{n^{\frac{|\wt(\rho_i)}{2}}MT_{u_i}^{\lambda}(\sigma_i)\right\}_{i\geq1}\overset{d}\to \left\{u_i\cdot \prod_{k\geq 2} k^{m_k(\rho_i)/2} \mathcal{H}_{m_k(\rho_i)}(\xi_k)\right\}_{i\geq1}.\]
\end{theorem}

Notice that the procedure of Proposition \ref{first order decomposition of the partial trace} can be iterated: consider $\sigma\in S_r$ and $s\in \mathbb{N}$ such that $n-r>s$. Set $\mu^{(s)}:=\lambda$ and $\mu^{(s-1)}:=\bar{\mu}$, where $\bar{\mu}$ is defined according to Proposition \ref{first order decomposition of the partial trace}. Similarly set $u^{(s)}:=u$ and $u^{(s-1)}:=\bar{u}$, then we can rewrite \eqref{first decomposition of partial trace} as 
\[PT_u^{\lambda}(\sigma)=MT_{u^{(s)}}^{\mu^{(s)}}(\sigma)+\frac{\dim\mu^{(s-1)}}{\dim\lambda}PT_{u^{(s-1)}}^{\mu^{(s-1)}}(\sigma).\]
Applying again Proposition \ref{first order decomposition of the partial trace} to the second term of the right hand side, we see that there exists $\mu^{(s-2)}\nearrow\mu^{(s-1)}$ and $u^{(s-2)}\in [0,1]$ such that 
\[PT_u^{\lambda}(\sigma)=MT_{u^{(s)}}^{\mu^{(s)}}(\sigma)+\frac{\dim\mu^{(s-1)}}{\dim\lambda}MT_{u^{(s-1)}}^{\mu^{(s-1)}}(\sigma)+\frac{\dim\mu^{(s-2)}}{\dim\lambda}PT_{u^{(s-2)}}^{\mu^{(s-2)}}(\sigma).\]
By iterating we obtain the following proposition:

\begin{proposition}\label{decomposition of the partial trace}
 Let $\sigma\in S_r$. Set $s$ such that $n-s>r$, then there exists a sequence of partitions $\mu^{(0)}\nearrow\mu^{(1)}\nearrow\ldots\nearrow \mu^{(s)}=\lambda$ and a sequence of real numbers $0\leq u_{0},\ldots,u_s<1$ such that
 \[PT_u^{\lambda}(\sigma)=\sum_{i=1}^{s}\frac{\dim\mu^{(i)}}{\dim\lambda}MT_{u^{(i)}}^{\mu^{(i)}}(\sigma)+\frac{\dim\mu^{(0)}}{\dim\lambda}PT_{u^{(0)}}^{\mu^{(0)}}(\sigma).\]
\end{proposition}

 \subsection{Partial permutations}
 In this section we recall some results in the theory of partial permutations, introduced in \cite{IvanovKerov1999}. All the definitions and results in this section can be found in \cite{feray2012partial}.
 \begin{defi}
  A \emph{partial permutation} is a pair $(\sigma,d)$, where $d\subset \mathbb{N}$ is finite and $\sigma$ is a bijection $d\to d$.
 \end{defi}
We call $P_n$ the set of partial permutations $(\sigma,d)$ such that $d\subseteq \{1,\ldots,n\}$. The set $P_n$ is endowed with the operation $(\sigma,d)\cdot(\sigma',d')=(\tilde{\sigma}\cdot\tilde{\sigma}',d\cup d')$, where $\tilde{\sigma}$ is the bijection from $d\cup d'$ to itself defined by $\tilde{\sigma}_{\vert d}=\sigma$ and $\tilde{\sigma}_{\vert d'\setminus d}=\id$; the same holds for $\tilde{\sigma}$.
\smallskip

Define the algebra $\mathcal{B}_n=\mathbb{C}[P_n]$. There is an action of the symmetric group $S_n$ on $P_n$ defined by $\tau\cdot (\sigma,d):=(\tau\sigma\tau^{-1},\tau(d))$, and we call $\mathcal{A}_n$ the abelian subalgebra of $\mathcal{B}_n$ of the invariant elements under this action. Set 
\[\alpha_{\rho;n}:=\sum\limits_{\substack{d\subseteq\{1,\ldots,n\},\rho\vdash|d|\\ \sigma\in S_d\mbox{ of type } \rho}}(\sigma,d).\]
\begin{proposition}
 The family $(\alpha_{\rho;n})_{|\rho|\leq n}$ forms a linear basis for $\mathcal{A}_n$.
\end{proposition}
There is a natural projection $\mathcal{B}_{n+1}\to\mathcal{B}_n$ which sends to $0$ the partial permutations whose support contains $n+1$. The projective limit through this projection is called $\mathcal{B}_{\infty}:=\lim\limits_{\leftarrow}\mathcal{B}_n$, and similarly $\mathcal{A}_{\infty}:=\lim\limits_{\leftarrow}\mathcal{A}_n$. The family $\alpha_{\rho}:=(\alpha_{\rho;n})_{n\geq 1}$ forms a linear basis of $\mathcal{A}_{\infty}$.
\smallskip

We recall the definition of \emph{Jucys-Murphy element}, described for example in \cite{Jucys1966} and \cite{Murphy1981}, and its generalization as a partial permutation. 
\begin{defi}
The $i$-th \emph{Jucys-Murphy element} is the element of the group algebra $\mathbb{C}[S_n]$ defined by $J_i:=(1,i)+(2,i)+\ldots+(i-1,i)$, $J_1=0$.
\smallskip

 The \emph{partial Jucys-Murphy element }$\xi_i$ is the element of $\mathcal{B}_n$ defined by $\xi_i:=\sum\limits_{j<i}((j,i),\{j,i\}),$ $\xi_1:=0.$
\end{defi}
\begin{proposition}
If $f$ is a symmetric function, $f(\xi_1,\ldots, \xi_n)\in\mathcal{A}_n$. The sequence $f_n=f(\xi_1,\ldots,\xi_n)$ is an element of the projective limit $\mathcal{A}_{\infty}$, that we denote $\Gamma_f$; moreover, $f\mapsto \Gamma_f$ is an algebra morphism. We  define $\Xi$ as the projective limit of the sequence $(\xi_1,\ldots,\xi_n)$ for $n\to\infty$; equivalently, $\Xi$ is the unique element of $\bigcup_{n=1}^{\infty}\mathbb{C}[S_n]$ such that $f(\Xi):=\Gamma_f$.
\end{proposition}
For a partition $\nu=(\nu_1,\ldots,\nu_q)$ and $n\geq |\nu|$ we consider \hbox{$p_{\nu}(\xi_1,\ldots,\xi_n)=\prod_{i=1}^q(\xi_1^{\nu_i}+\ldots +\xi_n^{\nu_i})$}. Then F\'eray proved in \cite[Proposition 2.5 and proof]{feray2012partial} that
\[p_{\nu}(\Xi)=\prod_i m_i(\nu)!\hspace{0.1cm}\alpha_{\nu+\underline{1}}+\sum\limits_{|\rho|<|\nu|+q}c_{\rho}\alpha_{\rho},\]
with $\nu+\underline{1}=(\nu_1+1,\ldots,\nu_q+1)$ and $c_{\rho}$ non-negative integers.

Let us fix some notation: we write $n^{\downarrow k}$ for the $k$-th falling factorial, that is, $n(n-1)\cdot\ldots\cdot(n-k+1)$, and $z_{\rho}=\prod_i \rho_i\prod_i m_i(\rho)!$ for the order of the centralizer of a permutation of type $\rho$. We call 
\[K_{\rho1^{n-|\rho|}}=\frac{1}{\#\{\sigma \mbox{ of cycle type }\rho1^{n-|\rho|}\}}\sum\sigma\in\C[S_n],\]
where the sum runs over the permutations $\sigma\in S_n$ of cycle type $\rho1^{n-|\rho|}$. We consider the morphism of algebras $\varphi_n\colon\mathcal{A}_n\to Z(\mathbb{C}[S_n])$ which sends $(\sigma,d)$ to $\sigma$. On the basis $(\alpha_{\rho;n})$ it acts as follows:
\begin{equation}\label{act of phi}
 \varphi_n(\alpha_{\rho;n})= \frac{n^{\downarrow|\rho|}}{z_{\rho}}K_{\rho1^{n-|\rho|}}.
\end{equation}
In particular $\hat{\chi}^{\lambda}(\varphi_n(\alpha_{\nu;n}))=\frac{1}{z_{\nu}}|\lambda|^{\downarrow |\nu|}\hat{\chi}^{\lambda}_{\nu}$. It follows from a result of Jucys\hbox{\cite[Equation (12)]{Jucys1974}} that
\begin{equation}\label{ciao}
 \hat{\chi}^{\lambda}(\varphi_n(f(\xi_1,\ldots,\xi_n)))=f(\mathcal{C}_{\lambda}),
\end{equation}
 where $f$ is a symmetric function and $\mathcal{C}_{\lambda}$ is the multiset of contents of the diagram $\lambda$, $\mathcal{C}_{\lambda}=\{c(\Box),\Box\in\lambda\}$.
 
 \subsection{Shifted symmetric functions}
 We recall some results on the algebra of shifted symmetric functions $\Lambda^{\ast}$ and their relations with $\mathcal{A}_{\infty}$. The results and definitions in this section follow \cite{IvanovKerov1999}.
 \begin{defi}
  Let $\Lambda^{\ast}(n)$ be the algebra of polynomials with complex coefficients in $x_1,\ldots,x_n$ that become symmetric in the new variables $x_i'=x_i-i$. There is a natural projection $\Lambda^{\ast}(n+1)\to \Lambda^{\ast}(n)$ which sends $x_{n+1}$ to $0$. The \emph{algebra of shifted symmetric functions} $\Lambda^{\ast}$ is defined as the projective limit of $\Lambda^{\ast}(n)$ according to this projection.
 \end{defi}
 We can apply (shifted) symmetric functions to partitions in the following way: if $f$ is a (shifted) symmetric function and $\lambda=(\lambda_1,\lambda_2,\ldots,\lambda_l)$ a partition, then we set $f(\lambda)=f(\lambda_1,\lambda_2,\ldots, \lambda_l)$.
 \smallskip
 
 As in the symmetric functions algebra, there are several interesting bases for the algebra of shifted symmetric functions. We present one of them:
 \begin{proposition}
There exists a family of shifted symmetric functions $\{p_{\rho}^{\sharp}\}_{\rho\vdash r}$ such that, for each $\lambda\vdash n$,
  \[p_{\rho}^{\sharp}(\lambda)=\left\{\begin{array}{lr}
                      n^{\downarrow r}\hat{\chi}_{\rho 1^{n-r}}^{\lambda},&\mbox{if }n\geq r;\\
                      0,&\mbox{otherwise}.
                     \end{array}\right.\]
Moreover, the family $\{p_{\rho}^{\sharp}\}$, where $\rho$ runs over all partitions, forms a basis for $\Lambda^{\ast}$.
 \end{proposition}
 These functions are called \emph{shifted power sums}; they were introduced in \cite[Section 15]{OkOl1998}.
 \smallskip
 
 In order to simplify the notation we write $p_k^{\sharp}$, where $k$ is a positive integer, instead of $p_{(k)}^{\sharp}$. Similarly, $p_{1^k}^{\sharp}$ replaces $p_{(1^k)}^{\sharp}$. As polynomials in the variables $x_1,\ldots,x_n,\ldots$ the function $p_{\rho}^{\sharp}$ have degree \hbox{$\deg p_{\rho}^{\sharp}(x_1,\ldots,x_n,\ldots)=|\rho|$}. We are interested in two filtrations of the algebra $\Lambda^{\ast}$, which are described in\hbox{\cite[section 10]{IvanovKerov1999}}:
\begin{itemize}
 \item $\deg(p_{\rho}^{\sharp})_1:=|\rho|_1:=|\rho|+m_1(\rho)$. This filtration is usually referred as the \emph{Kerov filtration};
 \item $\deg(p_{\rho}^{\sharp})_{\mathbb{N}}:=|\rho|_{\mathbb{N}}=|\rho|+l(\rho)$, where $l(\rho)=\sum_{i\in\mathbb{N}}m_i(\rho)$ is the number of parts of $\rho$.
\end{itemize}
 \begin{lemma}\label{lemma isomorphism}
  There is an isomorphism $F:\mathcal{A}_{\infty}\to\Lambda^{\ast}$ that sends $\alpha_{\rho}$ to $p_{\rho}^{\sharp}/z_{\rho}$.
 \end{lemma}
Our goal in this section is to develop \cite[Proposition 4.12]{ivanov2002kerov}, which gives information about the top degree term of $p_{\rho}^{\sharp}\cdot p_{\theta}^{\sharp}$ for partitions $\rho$ and $\theta$, where top degree refers to the Kerov filtration. Consequently, we obtain results on the product $\alpha_{\rho}\cdot\alpha_{\theta}$ (Corollary \ref{corollary on alpha}). 
\smallskip

We will write $V^{<}_k$ for the vector space whose basis is $\{ p^{\sharp}_{\rho}$ such that $|\rho|_1<k\}$. Notice that since $\deg(\cdot)_1$ is a filtration, 
 \begin{equation}\label{property of filtration}
p_{\rho}^{\sharp}\cdot V^{<}_k\subseteq V^{<}_{|\rho|_1+k}.  
\end{equation}
By abuse of notation we will often write $p_{\rho}^{\sharp}+V^{<}_k$ to indicate $p_{\rho}^{\sharp}$ plus a linear combination of elements with $\deg(\cdot)_1<k$. We recall a result of Ivanov and Olshanski \cite[Proposition 4.12]{ivanov2002kerov}:
\begin{lemma}\label{prop 4.12 of IO}
 For any partition $\rho$ and $k\geq 2$,
 \begin{equation}
  p_{\rho}^{\sharp}p_{k}^{\sharp}=p_{\rho\cup k}^{\sharp}+\left\{\begin{array}{lr}
                                                                  k\cdot m_k(\rho)\cdot p_{(\rho\setminus k)\cup 1^k}^{\sharp}+V^{<}_{|\rho|_1+k},&\mbox{if }m_k(\rho)\geq 1\\
                                                                  V^{<}_{|\rho|_1+k},&\mbox{if }m_k(\rho)=0,
                                                                 \end{array}\right.
 \end{equation}
where the partition $\rho\setminus k$ is obtained by removing one part equal to $k$.
\end{lemma}
\begin{lemma}
 For a partition $\rho$,
 \[p_{\rho}^{\sharp}p_{1^k}^{\sharp}=p_{\rho\cup 1^k}^{\sharp}+V^{<}_{|\rho|_1+2k}.\]
\end{lemma}
\begin{proof}
 We prove the statement by induction on $k$, the case $k=1$ being proved in\hbox{\cite[Proposition 4.11]{ivanov2002kerov}}. This implies that for any $k$
 \[p_{1^k}^{\sharp}p_1^{\sharp}=p_{1^{k+1}}^{\sharp}+V^{<}_{2k+2}.\]
Suppose the lemma true for $k>1$, then
\begin{align*}
 p_{\rho}^{\sharp}p_{1^{k+1}}^{\sharp}&=p_{\rho}^{\sharp}p_{1^k}^{\sharp}p_{1}^{\sharp}+p_{\rho}^{\sharp}V^{<}_{2k+2}\\
 &=p_{\rho\cup 1^k}^{\sharp}p_1^{\sharp}+V^{<}_{|\rho|_1+2k+2}\\
 &=p_{\rho\cup1^k\cup 1}^{\sharp}+V^{<}_{|\rho|_1+2k+2},
\end{align*}
where we repetitively used property \eqref{property of filtration}.
\end{proof}
Define $\tilde{V}^=_k$ as the vector space with basis $\{p_{\theta}^{\sharp}$ such that $|\theta|_1=k$ and $m_1(\theta)>0\}$, a slightly different statement as property \eqref{property of filtration} holds:
\begin{lemma}\label{property of filtration 2}
For a positive integer $k$ and a partition $\rho$
 \begin{equation}
p_{\rho}^{\sharp}\cdot \tilde{V}^=_k\subseteq \tilde{V}^=_{|\rho|_1+k}+ V^{<}_{|\rho|_1+k}.  
\end{equation}
\end{lemma}
\begin{proof}
 It is enough to show that, for a partition $\theta$ such that $|\theta|_1=k$ and $m_1(\theta)>0$, 
 \begin{equation}\label{this}
 p_{\rho}^{\sharp}\cdot p^{\sharp}_{\theta}\subseteq \tilde{V}^=_{|\rho|_1+|\theta|_1}+ V^{<}_{|\rho|_1+|\theta|_1}. 
 \end{equation}
Set $\theta=\tilde{\theta}\cup 1$. By the previous lemma  $p_{\theta}^{\sharp}=p_{\tilde{\theta}}^{\sharp}\cdot p_1^{\sharp}+V^<_{|\tilde{\theta}|_1+2}$ and 
\[p_{\rho}^{\sharp}\cdot p_{\theta}^{\sharp}=p_{\rho}^{\sharp}\cdot p_{\tilde{\theta}}^{\sharp}\cdot p_1^{\sharp}+V^<_{|\rho|_1+|\tilde{\theta}|_1+2}.\]
Set $C_{\rho,\tilde{\theta}}^{\tau}$ to be the structure constants in the basis $\{p_{\tau}^{\sharp}\}$ of the product $p_{\rho}^{\sharp}\cdot p_{\tilde{\theta}}^{\sharp}$, that is, 
\[p_{\rho}^{\sharp}\cdot p_{\tilde{\theta}}^{\sharp}=\sum_{|\tau|_1\leq |\rho|_1+|\tilde{\theta}|_1}C_{\rho,\tilde{\theta}}^{\tau}p_{\tau}^{\sharp},\]
where the restriction $|\tau|_1\leq |\rho|_1+|\tilde{\theta}|_1$ is a consequence of \eqref{property of filtration}. Thus 
\begin{align*}
p_{\rho}^{\sharp}\cdot p^{\sharp}_{\theta}&=\sum_{|\tau|_1\leq |\rho|_1+|\tilde{\theta}|_1}C_{\rho,\tilde{\theta}}^{\tau}p_{\tau}^{\sharp}\cdot p_1^{\sharp} +V^<_{|\rho|_1+|\tilde{\theta}|_1+2}\\
&=\sum_{|\tau|_1\leq |\rho|_1+|\tilde{\theta}|_1}C_{\rho,\tilde{\theta}}^{\tau}\left(p_{\tau\cup 1}^{\sharp}+V_{|\tau|_1+2}^<\right)+V^<_{|\rho|_1+|\tilde{\theta}|_1+2}\\
&\subseteq \tilde{V}^=_{|\rho|_1+|\tilde{\theta}|_1+2}+V^<_{|\rho|_1+|\tilde{\theta}|_1+2}=\tilde{V}^=_{|\rho|_1+|\theta|_1}+ V^{<}_{|\rho|_1+|\theta|_1}.
\end{align*}
This proves \eqref{this} and hence the lemma.
\end{proof}

\begin{proposition}
 Let $\rho$, $\theta$ be partitions. Then
 \[p_{\rho}^{\sharp}\cdot p_{\theta}^{\sharp}=p_{\rho\cup\theta}^{\sharp}+\tilde{V}^=_{|\rho|_1+|\theta|_1}+V^<_{|\rho|_1+|\theta|_1}.\]
\end{proposition}
\begin{proof}
 We prove the statement by induction on the number $l(\theta)-m_1(\theta)$ of non-one parts of $\theta$, with the initial case being $\theta=1^k$, shown in the previous lemma.
 \smallskip
 
 Consider now the claim true for $p_{\rho}^{\sharp}\cdot p_{\tilde{\theta}}^{\sharp}$ and set $\theta=\tilde{\theta}\cup k$ for $k\geq 2$. Then by Lemma \ref{prop 4.12 of IO}
 \begin{align*}
  p_{\theta}^{\sharp}=p_{\tilde{\theta}\cup k}^{\sharp}&=p_{\tilde{\theta}}^{\sharp}\cdot p_{k}^{\sharp}+\left\{\begin{array}{lr}
                                                                k\cdot m_k(\tilde{\theta})\cdot p_{(\tilde{\theta}\setminus k)\cup 1^k}^{\sharp}+V^<_{|\theta|_1},&\mbox{if }m_k(\tilde{\theta})\geq 1\\
                                                                  V^<_{|\theta|_1},&\mbox{if }m_k(\tilde{\theta})=0,
                                                                 \end{array}\right.\\
&=p_{\tilde{\theta}}^{\sharp}\cdot p_{k}^{\sharp}+\tilde{V}^=_{|\theta|_1}+V^<_{|\theta|_1}.
 \end{align*}
 Because of Property \eqref{property of filtration} and Lemma \eqref{property of filtration 2},
 \begin{equation}\label{bla}
 p_{\rho}^{\sharp}\cdot p_{\theta}^{\sharp}=p_{\rho}^{\sharp}\cdot p_{\tilde{\theta}}^{\sharp}\cdot p_k^{\sharp}+\tilde{V}^=_{|\rho|_1+|\theta|_1}+V^<_{|\rho|_1+|\theta|_1}. 
 \end{equation}
 By the inductive step 
 \[p_{\rho}^{\sharp}\cdot p_{\tilde{\theta}}^{\sharp}=p_{\rho\cup\tilde{\theta}}^{\sharp}+\tilde{V}^=_{|\rho|_1+|\tilde{\theta}|_1}+ V^<_{|\rho|_1+|\tilde{\theta}|_1}.\]
 Hence, by applying again Lemma \ref{prop 4.12 of IO}
  \begin{align*}
   p_{\rho}^{\sharp}\cdot p_{\tilde{\theta}}^{\sharp}\cdot p_k^{\sharp}&=p_{\rho\cup\tilde{\theta}}^{\sharp}\cdot p_k^{\sharp}+\tilde{V}^=_{|\rho|_1+|\tilde{\theta}|_1+k}+ V^<_{|\rho|_1+|\tilde{\theta}|_1+k}\\
   &=p_{\rho\cup\tilde{\theta}\cup k}^{\sharp}+\tilde{V}^=_{|\rho|_1+|\tilde{\theta}|_1+k}+ V^<_{|\rho|_1+|\tilde{\theta}|_1+k}.
  \end{align*}
We substitute the previous expression in \eqref{bla}
\[p_{\rho}^{\sharp}\cdot p_{\theta}^{\sharp}=p_{\rho\cup\tilde{\theta}\cup k}^{\sharp}+\tilde{V}^=_{|\rho|_1+|\tilde{\theta}|_1+k}+V^<_{|\rho|_1+|\tilde{\theta}|_1+k},\]
which concludes the proof.
\end{proof}
We can obtain a similar result in the algebra $\mathcal{A}_{\infty}$ by applying the isomorphism $F^{-1}$ defined in \ref{lemma isomorphism}. It is easy to see that $F^{-1}(\tilde{V}^=_k)$ is the space of linear combinations of $\alpha_{\rho}$ such that $|\rho|_1=k$ and $m_1(\rho)>0$. Similarly $F^{-1}(V^<_k)$ is the space of linear combinations of $\alpha_{\rho}$ such that $|\rho|_1<k$.
\begin{corollary}\label{corollary on alpha}
 Let $\rho$, $\theta$ be partitions. Then
 \[\alpha_{\rho}\cdot \alpha_{\theta}=\alpha_{\rho\cup\theta}+F^{-1}(\tilde{V}^=_k)+F^{-1}(V^<_k).\]
\end{corollary}

\section{Asymptotic of the main term for the partial trace}\label{Asymptotic of the main term}

In this section we prove Theorem \ref{convergence of transformed co-transition}. The main step is to prove that \hbox{$\hat{\chi}^{\lambda}_{\rho}-\hat{\chi}^{\mu}_{\rho}\in o_P(n^{-\frac{\wt(\rho)}{2}})$}, where $\rho$ is a fixed partition, $\wt(\rho)=|\rho|-m_1(\rho)$, $\lambda$ is a partition of $n$ and $\mu\nearrow \lambda$.
\bigskip

For a partition $\nu=(\nu_1,\ldots,\nu_q)$ we define slightly modified power sums $\tilde{p}_{\nu}$; we will show (Lemma \ref{lemma between mu and lambda} and Equation \eqref{ciao}) that 
\begin{equation}\label{ecco1}
 \hat{\chi}^{\lambda}(\varphi_n(\tilde{p}_{\nu}(\xi_1,\ldots,\xi_n)))-\hat{\chi}^{\mu}(\varphi_{n-1}(\tilde{p}_{\nu}(\xi_1,\ldots,\xi_{n-1})))=\tilde{p}_{\nu}(\mathcal{C}_{\lambda})-\tilde{p}_{\nu}(\mathcal{C}_{\mu})\in o_P(n^{\frac{|\nu|+q}{2}}).
\end{equation}
In order to translate this result on a bound on $\hat{\chi}^{\lambda}_{\rho}-\hat{\chi}^{\mu}_{\rho}$ we need to study the expansion of the modified power sums evaluated on $\Xi$, the infinite set of Jucys-Murphy elements in the theory of partial permutations:
\[ \tilde{p}_{\nu}(\xi_1,\ldots,\xi_n)=\sum_{\varsigma}c_{\varsigma}\alpha_{\varsigma;n}.\]
With the right choice of $\nu$ and the right filtration we will prove (Proposition \ref{final corollary}) that 
\begin{equation}\label{ecco2}
\tilde{p}_{\nu}(\mathcal{C}_{\lambda})-\tilde{p}_{\nu}(\mathcal{C}_{\mu})= \hat{\chi}^{\lambda}\circ\varphi_n\left(\sum_{\varsigma}c_{\varsigma}\alpha_{\varsigma;n}\right) -\hat{\chi}^{\mu}\circ\varphi_{n-1}\left(\sum_{\varsigma}c_{\varsigma}\alpha_{\varsigma;n-1}\right) =\hat{\chi}^{\lambda}_{\rho}-\hat{\chi}^{\mu}_{\rho}+ o_P(n^{-\frac{\wt(\rho)}{2}}).
\end{equation}
\bigskip
Comparing \eqref{ecco1} and \eqref{ecco2}, this gives 
\[\hat{\chi}^{\lambda}_{\rho}-\hat{\chi}^{\mu}_{\rho}\in o_P(n^{-\frac{\wt(\rho)}{2}}).\]
\begin{remark}
It is easy to see that $\hat{\chi}^{\lambda}_{\rho}-\hat{\chi}^{\mu}_{\rho}\in o_P(n^{-\frac{|\rho|-l(\rho)}{2}})$, where $l(\rho)$ is the number of parts of $\rho$; for example, considering balanced diagrams, in \cite{feray2007asymptotics} the authors prove an explicit formula for the normalized character, which implies $\hat{\chi}^{\lambda}_{\rho}-\hat{\chi}^{\mu}_{\rho}\in o_P(n^{-\frac{|\rho|-l(\rho)}{2}})$. Alternatively one can look at the descriptions of normalized characters expressed as polynomials in terms of free cumulants, studied by Biane in \cite{Biane2003}. The action of removing a box from a random partition $\lambda$ affects free cumulants in a sense described in \cite{dolkega2010explicit}, which furnish the aforementioned bound. On the other hand we needed a stronger result, namely Proposition \ref{final corollary}, and for this reason we introduce the modified power sums.
\end{remark}
\bigskip

Through the section, $\sigma$ will be a fixed permutation, $\rho$ its cycle type, and we consider $\nu=(\nu_1,\ldots,\nu_q)$, such that $\rho=\nu+\underline{1}=(\nu_1+1,\ldots,\nu_q+1,1\ldots,1)$.
 \subsection{Modified power sums}
 \begin{defi}
 For a partition $\nu=(\nu_1,\ldots,\nu_q)$ the \emph{modified power sum} is 
 \[\tilde{p}_{\nu}(x_1,x_2,\ldots)=\prod_{i=1}^q\left(p_{\nu_i}(x_1,x_2,\ldots)-\cat\left(\frac{\nu_i}{2}\right)\cdot\left(\frac{\nu_i}{2}\right)! \hspace{0.3cm}\alpha_{(1^{\frac{\nu_i}{2}+1})}\right),\]
 where $\cat(k)=\binom{2k}{k}\frac{1}{k+1}$ is the $k$-th Catalan number if $k$ is an integer, and $0$ otherwise.
 \end{defi}
 \bigskip
 
 \begin{lemma}\label{power sums on partial jucys murphy}
 Set $\nu=(\nu_1,\ldots,\nu_q)$, then
  \begin{equation}\label{p tilde}
   \tilde{p}_{\nu}(\Xi)=\sum_{\varsigma}c_{\varsigma}\alpha_{\varsigma},
  \end{equation}
  for some non-negative integers $c_{\varsigma}$. We have
  \begin{enumerate}
   \item the sum runs over the partitions $\varsigma$ such that $|\varsigma|_1\leq |\nu|_{\mathbb{N}}$;
   \item $c_{\nu+\underline{1}}=\prod_i (m_i(\nu)!)$;
   \item if $c_{\varsigma}\neq 0$, $|\varsigma|_1=|\nu|_{\mathbb{N}}$ and $\varsigma\neq \nu+\underline{1}$ then $m_1(\varsigma)>0$.
  \end{enumerate}
 \end{lemma}
\bigskip

\begin{proof}
\begin{proofpart}
Since $\deg(\alpha_{\varsigma})_1=|\varsigma|_1$ is a filtration, it is enough to prove the statement for $q=1$, that is, $\nu=(k)$ for some positive integer $k$. We consider
\begin{equation}\label{power}
 p_k(\Xi)=\sum_{1\leq j_1,\ldots,j_k<h}\left((j_1,h),\{j_1,h\}\right)\cdot\ldots\cdot\left((j_k,h),\{j_k,h\}\right).
\end{equation}
Let $\left(\sigma,d\right)=\left((j_1,h)\cdot\ldots\cdot(j_k,h),\{j_1,\ldots,j_k,h\}\right)$ be a term of the previous sum and let $\varsigma\vdash |d|$ be the cycle type of $\sigma$. This is the outline of the proof: firstly, we show an upper bound for $|\varsigma|_1,$ that is, $|\varsigma|_1\leq k+2$. Then we check some conditions that $\sigma$ must satisfy in order to have $|\varsigma|_1=k+2$ (if such a $\sigma$ exists). We prove that $|\varsigma|_1= k+2$ if and only if $(\sigma,d)=(\id,d)$ with $|d|=k/2+1$ and $k $ even. Finally, we calculate how many times the partial permutation $(\id,d)$, with $|d|=k/2+1$, appears in the modified power sum $p_k(\Xi)$. This number is the coefficient of $\alpha_{(1^{\frac{k}{2}+1})}$ in $p_k(\Xi)$, and we show that it is equal to $\cat\left(\frac{k}{2}\right)\cdot\left(\frac{k}{2}\right)!$. Since we defined $\tilde{p}_{k}(\Xi)=p_{k}(\Xi)-\cat\left(\frac{k}{2}\right)\cdot\left(\frac{k}{2}\right)!\cdot\alpha_{(1^{\frac{k}{2}+1})}$ this shows that all the terms in $\tilde{p}_k(\Xi)$ satisfy $|\varsigma|_1\leq k+1=|\nu|_{\mathbb{N}}$ when $\nu=(k)$.
\medskip 

We first estimate
\begin{align*}
 |\varsigma|_1&=|\{j_1,\ldots,j_k,h\}|+m_1(\varsigma)\\
 &= 2\#\{j_i, \mbox{ s.t. } \sigma j_i=j_i\}+\#\{j_i, \mbox{ s.t. } \sigma j_i\neq j_i\}+2\delta(h\mbox{ is fixed by }\sigma)+\delta(h\mbox{ is not fixed by }\sigma),
\end{align*}
where $\delta$ is the Kronecker delta. We stress out that, when counting the set cardinalities above, we do not count multiplicities; for example, if 
\[\sigma=(1,5)(2,5)(2,5)(1,5)(3,5)=(1)(2)(3,5),\qquad\mbox{ then}\]
\[\#\{j_i, \mbox{ s.t. } \sigma j_i=j_i\}=\#\{1,2\}=2,\qquad \#\{j_i, \mbox{ s.t. } \sigma j_i\neq j_i\}=\#\{3\}=1,\qquad\varsigma=(2,1,1).\]
Notice that in the sequence $(j_1,\ldots,j_k)$ all fixed points must appear at least twice, while non fixed points must appear at least once, hence
\[2\#\{j_i, \mbox{ s.t. } \sigma j_i=j_i\}+\#\{j_i, \mbox{ s.t. } \sigma j_i\neq j_i\}\leq k,\]
while obviously,
\[2\delta(h\mbox{ is fixed by }\sigma)+\delta(h\mbox{ is not fixed by }\sigma)\leq 2.\]
Thus $|\varsigma|_1\leq k+2.$
\medskip

Suppose there exists $\sigma$ such that $|\varsigma|_1=k+2$, then from the proof of the inequality $|\varsigma|_1\leq k+2$ we know that $\sigma$ satisfies
\begin{equation}\label{property at most twice}
 \mbox{for each }i,\mbox{ }j_i\mbox{ appears at most twice;}
\end{equation}
\begin{equation}\label{property of exactly twice}
 j_i \mbox{ is fixed by }\sigma\mbox{ iff it appears exactly twice in the multiset }\{j_1,\ldots,j_k\};
\end{equation}
\begin{equation}\label{property of h fixed}
 h\mbox{ is a fixed point}.
\end{equation}
 We prove by induction on $k$ that if $|\varsigma|_1=k+2$ and $\alpha_{\varsigma}$ appears in the sum \eqref{power} then $k$ is even and $\varsigma=(1^{\frac{k}{2}+1})$. If $k=1$ then $\varsigma=(2)$ and $|\varsigma|_1=2<3=k+2$. If $k=2$ then $\varsigma$ can be either $\varsigma=(3)$ or $\varsigma=(1,1)$. By our request that $|\varsigma|_1=4=k+2$ we see that we must have $\varsigma=(1,1)$. Consider now the statement true up to $k-1$ and $\sigma=(j_1,h)\cdot\ldots\cdot(j_k,h)$. By property \eqref{property of h fixed} $h$ is fixed, so that $j_k$ must appear at least twice, and by \eqref{property at most twice} $j_k$ appears exactly twice. Hence, by \eqref{property of exactly twice} $j_k$ is fixed, thus it exists a unique $l<k-1$ such that $j_{l+1}=j_k$ and 
\[\sigma=(j_1,h)\ldots(j_l,h)\cdot(j_k,h)\cdot(j_{l+2},h)\ldots(j_{k-1},h)(j_k,h)=\tau_1(j_k,h)\tau_2(j_k,h)=\tau_1\tau_2,\]
where 
\[\tau_1=\left\{\begin{array}{lr}   \id &\mbox{if }l=0\\(j_1,h)\ldots(j_l,h)&\mbox{if }l>0,        \end{array}\right.\qquad\quad\tau_2=\left\{\begin{array}{lr}   \id &\mbox{if }l=k-2\\(j_{l+2},j_k)\ldots(j_{k-1},j_k)&\mbox{if }l<k-2.        \end{array}\right.\]
If $l=0$ or $l=k-2$ we can apply induction on either $\tau_1$ or $\tau_2$ and the result follows, so consider $0<l<k-2$. We claim that the sets $\{j_1,\ldots j_l\}$ and $\{j_{l+2},\ldots,j_k\}$ are disjoint. Notice first that $j_k$ does not appear in $\{j_1,\ldots j_l\}$. Suppose $j_a=j_b$ for some $a\leq l<b$, and choose $b$ minimal with this property. Then $j_b$ is not fixed by $\tau_2$, that is, $\tau_2(j_b)=j_{\tilde{b}}$ with either $\tilde{b}<b$ or $\tilde{b}=k$. In the first case ($\tilde{b}<b$), by minimality of $b$, $j_{\tilde{b}}$ does not appear in $\tau_1$; in the second case ($\tilde{b}=k$) we know that $j_k$ does not appear in $\{j_1,\ldots j_l\}$. Hence in either case $\sigma(j_b)=j_{\tilde{b}}$. This is a contradiction, since $j_b$ appears twice in $\sigma$ and therefore must be fixed (property \eqref{property of exactly twice}). This proves that $\{j_1,\ldots j_l\}\cap\{j_{l+2},\ldots,j_k\}=\emptyset$.

Therefore $\tau_1$ and $\tau_2$ respect properties \eqref{property at most twice},\eqref{property of exactly twice} and \eqref{property of h fixed}, and we can apply the inductive hypothesis to obtain that $\tau_1=\id=\tau_2$ and both $l$ and $k-l-2$ are even. We conclude hence that if $|\varsigma|_1=k+2$ and $\alpha_{\varsigma}$ appears in the sum \eqref{p tilde} then $\varsigma=(1^{\frac{k}{2}+1})$ and $k$ even. 
\medskip 

We assume now that $k$ is even and we calculate the coefficient of $\alpha_{(1^{\frac{k}{2}+1})}$ in $p_k(\Xi)$, which is equal to the coefficient of $\left(\id,d\right)=\left((j_1,h)\cdot\ldots\cdot(j_k,h),\{j_1,\ldots,j_k,h\}\right)$ in the sum \eqref{power}, for a fixed $d$ of cardinality \hbox{$k/2+1$}. In order to do this we count the number of lists $L=\left((j_1,h),\ldots,(j_k,h)\right)$ such that \hbox{$\{j_1,\ldots,j_k,h\}=d$} and $(j_1,h)\ldots(j_k,h)=\id$. We call $\mathcal{L}_d$ the set of these lists.
\smallskip

Define a \emph{pair set partition} of the set $[k]=\{1,\ldots,k\}$ as a set partition $A=\{\{r_1,s_1\},\ldots,\{r_{\frac{k}{2}},s_{\frac{k}{2}}\}\}$ of $[k]$ into pairs. Such a set partition is said to be \emph{crossing} if $r_a<r_b<s_a<s_b$ for some $a,b\leq k/2$, otherwise the set partition is said to be \emph{non crossing}. Calling $\mathcal{S}_k$ the set of non crossing pair set partitions, it is known that $|\mathcal{S}_k|=\cat(k/2)$, see \cite{hora1998central}. We build a map $\psi_k\colon \mathcal{L}_d\to\mathcal{S}_k$ and we prove that this map is $\left(\frac{k}{2}\right)!$-to-one, which implies that $|\mathcal{L}_d|=c_{(1^{\frac{k}{2}+1})}=\cat\left(\frac{k}{2}\right)\cdot\left(\frac{k}{2}\right)!$.
\smallskip

Let $L\in\mathcal{L}_d$, $L=(L_1,\ldots,L_k)=((j_1,h),\ldots,(j_k,h))$; we have proven that, since $(j_1,h)\cdot\ldots\cdot(j_k,h)=\id$ and $|d|=|\{j_1,\ldots,j_k,h\}|=k/2+1$, then each element $(j_i,h)$ must appear exactly twice in $L$. We construct a pair partition $\psi_k(L)$ such that a pair $\{r,s\}\in \psi_k(L)$ iff $j_r=j_s$. By \cite[Lemma 2]{hora1998central} this set partition is non crossing. This map is clearly surjective, although not injective: every permutation $\gamma$ acting on $d=\{j_1,\ldots,j_k,h\}$ which fixes $h$ acts also on $\mathcal{L}_d$: $\gamma(L)=((\gamma(j_1),h),\ldots,(\gamma(j_k),h))$. Notice that $h>j_1,\ldots,j_k$ and this is why $\gamma(h)=h$ in order to have an action on $\mathcal{L}_d$. Moreover $\psi_k(L)=\psi_k(L')$ if and only if $L=\gamma(L')$ for some $\gamma$ in $S_{d\setminus h}$. Thus $\psi_k$ is a $\left(\frac{k}{2}\right)!$-to-one map, hence $c_{(1^{\frac{k}{2}+1})}=\cat\left(\frac{k}{2}\right)\cdot\left(\frac{k}{2}\right)!$ and the first part of the proof is concluded.
\end{proofpart}
\begin{proofpart}
 The second statement is shown in \cite[section 2]{feray2012partial}.
\end{proofpart}
\begin{proofpart}
 This claim is proven by induction on $q$, the number of parts of $\nu$. The initial case is $q=1$ and
 \[\tilde{p}_k(\Xi)=\sum_{|\varsigma|_1\leq k+1}c_{\varsigma}\alpha_{\varsigma}.\]
 We consider the $\alpha_{\varsigma}$'s appearing in the sum such that $|\varsigma|_1=k+1$ and $m_1(\varsigma)=0$, so that $|\varsigma|=k+1$. Expanding the sum in a way similar to Equation \eqref{power}, we see that $|\varsigma|=k+1=\#\{j_1,\ldots,j_k,h\}$, and all the elements in this set must be pairwise different. Hence $\sigma=(j_1,h)\cdot\ldots\cdot (j_k,h)=(h,j_k,\ldots,j_1)$ and $\varsigma=(k+1)$. Thus the statement for $q=1$ is proved.
 
 Let $\nu=(\nu_1,\ldots,\nu_q)$ with $q\geq 2$. Set $\tilde{\nu}=(\nu_1,\ldots,\nu_{q-1})$ and suppose, by the induction hypothesis, that the assertion is true for $\tilde{p}_{\tilde{\nu}}(\Xi)$. Then
 \begin{align*}
  \tilde{p}_{\nu}(\Xi)&=\tilde{p}_{\tilde{\nu}}(\Xi)\tilde{p}_{\nu_q}(\Xi)\\
  &=\left(\sum_{|\tilde{\varsigma}|_1\leq |\tilde{\nu}|_{\mathbb{N}}}c_{\tilde{\varsigma}}\alpha_{\tilde{\varsigma}}\right)\cdot\left(\sum_{|\theta|_1\leq \nu_q+1}c_{\theta}\alpha_{\theta}\right)\\
  &=\sum_{\substack{|\tilde{\varsigma}|_1\leq |\tilde{\nu}|_{\mathbb{N}}\\|\theta|_1\leq \nu_q+1}}c_{\tilde{\varsigma}}c_{\theta}\cdot \alpha_{\tilde{\varsigma}}\alpha_{\theta}.
 \end{align*}
We apply now Corollary \ref{corollary on alpha} and obtain
\[\tilde{p}_{\nu}(\Xi)=\sum_{\substack{|\tilde{\varsigma}|_1\leq |\tilde{\nu}|_{\mathbb{N}}\\|\theta|_1\leq \nu_q+1}}c_{\tilde{\varsigma}}c_{\theta}\alpha_{\tilde{\varsigma}\cup\theta}+F^{-1}(\tilde{V}^=_{|\nu|_{\mathbb{N}}})+F^{-1}(V^<_{|\nu|_{\mathbb{N}}}).\]
 Hence there exists only one term in the previous sum such that $|\varsigma|_1=|\tilde{\varsigma}\cup\theta|_1=|\nu|_{\mathbb{N}}$ and $m_1(\varsigma)=0$, that is, $\varsigma=\nu+\underline{1}$, and the proof is completed. \qedhere
\end{proofpart}
\end{proof}
Let now $\lambda\vdash n$ be a random partition distributed with the Plancherel measure. We say that a function $X$ on the set of partitions is $X\in o_P(n^{\beta})$ if $n^{-\beta}X(\lambda)\overset{p}\to 0$. Similarly, $X\in O_P(n^{\beta})$ is \emph{stochastically bounded} by $n^{\beta}$ if for any $\epsilon>0$ there exists $M>0$ such that $P_{PL}(|X(\lambda)n^{-\beta}|>M)\leq \epsilon$. For example, as a consequence of Kerov's result on the convergence of characters (Theorem \ref{convergence of characters}), $\hat{\chi}^{\lambda}_{\rho}\in O_P(n^{-wt(\rho)/2})$. If a function $X(\lambda,\mu_j)$ depends also on a subpartition $\mu_j\nearrow\lambda$, by $X(\lambda,\mu_j)\in o_P(n^{\beta})$ we mean that $\max_j X(\lambda,\mu_j)\in o_P(n^{\beta})$, and similarly for the notion of stochastic boundedness.

\begin{lemma}\label{newlemma}
 For each partition $\rho$ we have that $\hat{\chi}^{\lambda}(\varphi_n(\alpha_{\rho;n}))\in O_P( n^{\frac{|\rho|_1}{2}})$, where $|\rho|_1=|\rho|+m_1(\rho)$.
\end{lemma}
\begin{proof}
From Equation \eqref{act of phi}, $\hat{\chi}^{\lambda}(\varphi_n(\alpha_{\rho;n}))=\frac{c_{\rho}}{z_{\rho}}n^{\downarrow|\rho|}\hat{\chi}^{\lambda}_{\rho}$. Since $\hat{\chi}^{\lambda}_{\rho}\in O_P(n^{-\wt(\rho)/2})$ the lemma follows.
\end{proof}

 \begin{lemma}\label{lemma between mu and lambda}
Let $\mu_j\nearrow\lambda$. In probability
  \[\hat{\chi}^{\lambda}(\varphi_n(\tilde{p}_{\nu}(\xi_1,\ldots,\xi_n)))-\hat{\chi}^{\mu_j}(\varphi_{n-1}(\tilde{p}_{\nu}(\xi_1,\ldots,\xi_{n-1})))=\tilde{p}_{\nu}(\mathcal{C}_{\lambda})-\tilde{p}_{\nu}(\mathcal{C}_{\mu_j})\in o_P(n^{\frac{|\nu|+q}{2}}).\]
 \end{lemma}
\begin{proof}
The first equality come from the fact that $\hat{\chi}^{\lambda}\circ\varphi_n$ applied to a symmetric function of partial Jucys-Murphy elements equals the same symmetric function evaluated on the contents of $\lambda$. We prove now the asymptotic bound. Consider first the case $q=1$, {\it i.e.}, $\nu=(k)$ for some positive integer $k$, in which case
 \begin{multline*}
  \tilde{p}_k(\mathcal{C}_{\lambda})=\tilde{p}_k(\mathcal{C}_{\mu_j}\cup y_j)=\tilde{p}_k(\mathcal{C}_{\mu_j})+y_j^k+\frac{\cat(\frac{k}{2})}{\frac{k}{2}+1}\left(n^{\downarrow(\frac{k}{2}+1)}-(n-1)^{\downarrow(\frac{k}{2}+1)}\right)\\
  =\tilde{p}_k(\mathcal{C}_{\mu_j})+y_j^k+\cat\left(\frac{k}{2}\right)(n-1)^{\downarrow\frac{k}{2}}.
 \end{multline*}
 Notice that $y_j<\max\{\lambda_1,\lambda_1'\}$ since $y_j$ is the content of a box of $\lambda$. It is shown in \cite[Lemma 1.5]{romik2015surprising} that, with probability that goes to $1$, both $\lambda_1$ and $\lambda_1'$ are smaller than $3\sqrt{n}$, hence for each subpartition $\mu_j$ of $\lambda$, one has $y_j<3\sqrt{n}$ with probability tending to $1$ for $n$ increasing. Therefore we obtain that $ \tilde{p}_k(\mathcal{C}_{\lambda})-\tilde{p}_k(\mathcal{C}_{\mu_j})\in O_P(n^{\frac{k}{2}})\in o_P(n^{\frac{k+1}{2}})$.
 
 In the general case $q>1$ we expand the product, obtaining
  \[\tilde{p}_{\nu}(\mathcal{C}_{\lambda})-\tilde{p}_{\nu}(\mathcal{C}_{\mu_j})=\sum_{A\subsetneq\{1,\ldots,q\}}\prod_{i\in A}\tilde{p}_{\nu_i}(\mathcal{C}_{\mu_j})\prod_{i\notin A}\left(y_j^{\nu_i}+\cat\left(\frac{\nu_i}{2}\right)(n-1)^{\downarrow\frac{\nu_i}{2}}\right).  \]
We use now Lemma \ref{power sums on partial jucys murphy} and \ref{newlemma}, which show that the factor $\prod_{i\in A}\tilde{p}_{\nu_i}(\mathcal{C}_{\mu_j})$ is in $O_P(n^{\frac{1}{2}(\sum_{i\in A}\nu_i+|A|)})$ and 
\[\prod_{i\notin A}\left(y_j^{\nu_i}+\cat\left(\frac{\nu_i}{2}\right)(n-1)^{\downarrow\frac{\nu_i}{2}}\right)\in O_P(n^{\frac{1}{2}\sum_{i\notin A}\nu_i}).\]
 Therefore $\tilde{p}_{\nu}(\mathcal{C}_{\lambda})-\tilde{p}_{\nu}(\mathcal{C}_{\mu_j})\in O_P(n^l)$, with 
 \[l=\max_{A\subsetneqq\{1,\ldots,q\}}\frac{1}{2}\sum_i\nu_i+\frac{|A|}{2}<\frac{|\nu|+q}{2}.\qedhere\] 
\end{proof}
\begin{proposition}\label{final corollary}
Let $\lambda\vdash n$, $\mu\nearrow\lambda$. Let $\rho\vdash r<n$  then 
\[\hat{\chi}^{\lambda}_{\rho}-\hat{\chi}^{\mu}_{\rho}\in o_P(n^{-\frac{\wt(\rho)}{2}}).\]
\end{proposition}
\begin{proof}
Set $\nu=(\nu_1,\ldots, \nu_q)$ such that $\rho=(\nu_1+1,\ldots,\nu_q+1,1,\ldots,1)$. obviously, $\hat{\chi}^{\lambda}_{\nu+\underline{1}}=\hat{\chi}^{\lambda}_{\rho}$ and $\hat{\chi}^{\mu}_{\nu+\underline{1}}=\hat{\chi}^{\mu}_{\rho}$; moreover, $|\nu|+q=\wt(\rho)$. We want to prove that $\hat{\chi}^{\lambda}_{\nu+\underline{1}}-\hat{\chi}^{\mu}_{\nu+\underline{1}}\in o_P(n^{-\frac{|\nu|+q}{2}})$. We prove the statement by induction on $|\nu+\underline{1}|_1=|\nu|+q$. The initial case is $\nu+\underline{1}=(1)$ and $\hat{\chi}^{\lambda}_{(1)}=1=\hat{\chi}^{\mu}_{(1)}$, so the proposition is trivially true.
 
  Consider $n^{-\frac{|\nu|+q}{2}}\left(\tilde{p}_{\nu}(\mathcal{C}_{\lambda})-\tilde{p}_{\nu}(\mathcal{C}_{\mu})\right)\in o_P(1)$ because of the previous lemma. It can be rewritten as
\[n^{-\frac{|\nu|+q}{2}}\sum_{\substack{\varsigma\mbox{ s.t.}\\|\varsigma|_1\leq |\nu|+q}}\left(\hat{\chi}^{\lambda}\circ\varphi_n-\hat{\chi}^{\mu}\circ\varphi_{n-1}\right)(  c_{\varsigma}\alpha_{\varsigma;n})\overset{p}\to 0,\]
where the coefficients $c_{\varsigma}$ are described in Lemma \ref{power sums on partial jucys murphy}.
\smallskip

 Equivalently
 \begin{align*}
  n^{-\frac{|\nu|+q}{2}}\left(\tilde{p}_{\nu}(\mathcal{C}_{\lambda})-\tilde{p}_{\nu}(\mathcal{C}_{\mu})\right)&=
  n^{-\frac{|\nu|+q}{2}}\sum_{|\varsigma|_1\leq |\nu|+q}\frac{c_{\varsigma}}{z_{\varsigma}}\left(n^{\downarrow|\varsigma|}\hat{\chi}^{\lambda}_{\varsigma}-(n-1)^{\downarrow|\varsigma|}\hat{\chi}^{\mu}_{\varsigma}\right)\\
  &=n^{-\frac{|\nu|+q}{2}}\sum_{|\varsigma|_1\leq |\nu|+q}\frac{c_{\varsigma}}{z_{\varsigma}}\left(\hat{\chi}^{\lambda}_{\varsigma}(n^{\downarrow|\varsigma|}-(n-1)^{\downarrow|\varsigma|})+(n-1)^{\downarrow|\varsigma|}(\hat{\chi}^{\lambda}_{\varsigma}-\hat{\chi}^{\mu}_{\varsigma})\right).
  \end{align*}
  We split the previous sum and notice that $n^{-\frac{|\nu|+q}{2}}\sum_{\varsigma}\frac{c_{\varsigma}}{z_{\varsigma}}|\varsigma|(n-1)^{\downarrow(|\varsigma|-1)}\hat{\chi}^{\lambda}_{\varsigma}\overset{p}\to 0$. Indeed 
  \[(n-1)^{\downarrow(|\varsigma|-1)}\cdot n^{-\frac{|\nu|+q}{2}}\leq n^{\frac{|\nu|+q}{2}-1}\qquad \mbox{and}\qquad n^{\frac{|\nu|+q}{2}-1} \hat{\chi}^{\lambda}_{\varsigma}\overset{p}\to 0\]
  since $|\varsigma|_1\leq |\nu|+q$.
  
  We deal with the sum $n^{-\frac{|\nu|+q}{2}}\sum_{\varsigma}\frac{c_{\varsigma}}{z_{\varsigma}}\left((n-1)^{\downarrow|\varsigma|}(\hat{\chi}^{\lambda}_{\varsigma}-\hat{\chi}^{\mu}_{\varsigma})\right)$. We separate in this sum the terms with $|\varsigma|_1=|\nu|+q$ and $\varsigma\neq \nu+\underline{1}$, the terms with $|\varsigma|_1<|\nu|+q$, and the term corresponding to $\varsigma=\nu+\underline{1}$.
\begin{itemize}
 \item Case $|\varsigma|_1=|\nu|+q$ and $\varsigma\neq \nu+\underline{1}$: we want to estimate 
 \[n^{-\frac{|\nu|+q}{2}}\sum_{\substack{|\varsigma|_1=|\nu|+q\\m_1(\varsigma)>0}}\frac{c_{\varsigma}}{z_{\varsigma}}\left((n-1)^{\downarrow|\varsigma|}(\hat{\chi}^{\lambda}_{\varsigma}-\hat{\chi}^{\mu}_{\varsigma})\right),\]
where the restriction $m_1(\varsigma)>0$ is a consequence of Lemma \ref{power sums on partial jucys murphy}, part 3. We consider one term of the previous sum and we write $\tilde{\nu}:=\varsigma-\underline{1}$, removed of the $0$ parts. Notice as before that $\hat{\chi}^{\lambda}_{\varsigma}= \hat{\chi}^{\lambda}_{\tilde{\nu}+\underline{1}}$, and $\hat{\chi}^{\mu}_{\varsigma}= \hat{\chi}^{\mu}_{\tilde{\nu}+\underline{1}}$. Thus $|\tilde{\nu}+\underline{1}|<|\varsigma|_1=|\nu|+q$ and we can apply the induction hypothesis. Therefore $\hat{\chi}^{\lambda}_{\tilde{\nu}+\underline{1}}-\hat{\chi}^{\mu}_{\tilde{\nu}+\underline{1}}\in o_P(n^{-\frac{|\tilde{\nu}+\underline{1}|}{2}})$ and
\[n^{-\frac{|\varsigma|_1}{2}}\cdot (n-1)^{\downarrow|\varsigma|}(\hat{\chi}^{\lambda}_{\varsigma}-\hat{\chi}^{\mu}_{\varsigma})\in o_P\left(n^{\frac{|\varsigma|_1}{2}-\frac{m_1(\varsigma)}{2}-\frac{|\tilde{\nu}+\underline{1}|}{2}}\right)=o_P(1).\]
 \item Case $|\varsigma|_1<|\nu|+q$. We can apply induction again, to have $(\hat{\chi}^{\lambda}_{\varsigma}-\hat{\chi}^{\mu}_{\varsigma})\in o_P(n^{-\frac{|\varsigma|-m_1(\varsigma)}{2}})$. Therefore
 \[n^{-\frac{|\nu|+q}{2}}\cdot (n-1)^{\downarrow|\varsigma|}(\hat{\chi}^{\lambda}_{\varsigma}-\hat{\chi}^{\mu}_{\varsigma})\in o_P\left(n^{-\frac{|\nu|+q}{2}+\frac{|\varsigma|_1}{2}}\right)\subseteq o_P(1).\]
\end{itemize}
We obtain thus that 
\[n^{-\frac{|\nu|+q}{2}}\cdot (n-1)^{|\nu|+q}\frac{\prod_i m_i(\nu)!}{z_{\nu+\underline{1}}}(\hat{\chi}^{\lambda}_{\nu+\underline{1}}-\hat{\chi}^{\mu}_{\nu+\underline{1}})+o_P(1)\overset{p}\to 0,\]
which proves the statement.
\end{proof}
 \begin{proof}[Proof of Theorem \ref{convergence of transformed co-transition}]
Fix $\sigma\in S_r$ and let $\rho$ be its cycle type. We have
\[ MT^{\lambda}_u(\sigma)=\sum_{ y_j\leq v^{\lambda}\sqrt{n}}\frac{\dim\mu_j}{\dim\lambda}\hat{\chi}^{\mu_j}_{\rho}=\sum_{ y_j\leq v^{\lambda}\sqrt{n}}\frac{\dim\mu_j}{\dim\lambda}\hat{\chi}^{\lambda}_{\rho}+o_P(n^{-\frac{|\nu|+q}{2}}),\]
since $\sum \dim\mu_j/\dim\lambda\leq 1$, and the previous proposition. Hence
\[n^{\frac{|\nu|+q}{2}} MT^{\lambda}_u(\sigma)=n^{\frac{|\nu|+q}{2}}\sum_{ y_j\leq v^{\lambda}\sqrt{n}}\frac{\dim\mu_j}{\dim\lambda}\hat{\chi}^{\lambda}_{\rho}+o_P(1).\]
Finally, by Lemma \ref{behaviour of co-transition} and Theorem \ref{convergence of characters}, we obtain that given $\sigma_1,\sigma_2,\ldots$ permutations of cycle type respectively $\rho_1,\rho_2,\ldots$ and $u_1,u_2,\ldots\in [0,1]$ and calling $\{\xi_k\}_{k\geq 2}$ a family of independent standard Gaussian variables 
 \[\left\{n^{\frac{|\rho_i|-m_1(\rho_i)}{2}}MT_{u_i}^{\lambda}(\sigma_i)\right\}=\left\{n^{\frac{|\nu_i|+q_i}{2}}\sum_{ y_j\leq v_i^{\lambda}\sqrt{n}}\frac{\dim\mu_j}{\dim\lambda}\hat{\chi}_{\rho_i}^{\lambda}+o_P(1)\right\}
 \overset{d}\to \left\{u_i\cdot \prod_{k\geq 2} k^{m_k(\rho_i)/2} \mathcal{H}_{m_k(\rho_i)}(\xi_k)\right\},\]
 for $i\geq 1$.
\end{proof}

\section{Sum of the entries of an irreducible representation}

In this chapter our goal is to describe the sum of the entries of the matrix associated to a Young seminormal representation, up to a certain index (depending on the dimension of the representation). 
We stress out that the objects we study really depend on the representation matrix, and change, for example, under isomorphisms of the representation. 
Some calculations are similar to those in the previous chapter: first we consider the sum of all the entries in the matrix (before this role was played by the trace), and then we study the sum of the entries whose indices $(i,j)$ satisfy \hbox{$i\leq u\dim\lambda$}, \hbox{$j\leq u\dim\lambda$}, while before we were considering the partial trace.
\subsection{Total sum}
\begin{defi}
Define the normalized total sum associated to an irreducible representation and a permutation $\sigma\in S_r$ as
\[TS^{\lambda}(\sigma):=\sum_{i,j\leq \dim\lambda} \frac{\pi^{\lambda}(\sigma)_{i,j}}{\dim\lambda}\]
\end{defi}
The following is the main result of the section:
\begin{theorem}\label{convergence of total sum}
Fix $\sigma_1\in S_{r_1},\sigma_2\in S_{r_2},\ldots$ and let $\lambda\vdash n$. Define the real numbers
\[m_i:= \mathbb{E}_{PL}^{r_i}\left[TS^{\nu}(\sigma_i)\right]\qquad \mbox{and}\qquad v_i:=\binom{r_i}{2}\mathbb{E}_{PL}^{r_i}\left[\hat{\chi}^{\nu}_{(2,1,\ldots,1)} TS^{\nu}(\sigma_i)\right].\]
 Then
 \[\left\{\sqrt{n}\cdot(TS^{\lambda}(\sigma_i)-m_i)\right\}\overset{d}\to\left\{ \mathcal{N}(0,2 v_i^2)\right\},\]
 where $\mathcal{N}(0,2 v_i^2)$ is a normal random variable of variance $2 v_i^2$.
\end{theorem}
Note that the limit $\mathcal{N}(0,2 v^2)$ is degenerate in the case $v=0$. Whether or not $v=0$ is a nontrivial question. At the end of the section we study the values of $v$ and $m$ when $\sigma$ is an adjacent transposition, and we write the explicit values of $v$ and $m$ for permutations of $S_4$. In order to prove the Theorem, we need some preliminary results.

\begin{proposition}
Let $\lambda\vdash n$ and $r\leq n$. There exists a bijection $\phi_r$ between
 \[\SYT(\lambda)\overset{\phi_r}\simeq \bigsqcup_{\nu\vdash r} \SYT(\nu)\times \SYT(\lambda/\nu)\]
\end{proposition}
\begin{proof}
 Let $T$ be a standard Young tableau of shape $\lambda$, the image $\phi_r(T)=(U,V)$ of $T$ is defined as follows: the boxes of $T$ whose entries are smaller or equal than $r$ identify a tableau $U$ of shape $\nu\subseteq \lambda$. Define now $V$ as a standard Young tableau of skew shape $\lambda/\nu$, and in each box write $a-r$, where $a$ is the value inside the corresponding box in $T$.
\end{proof}

\begin{example}
Here is an example for $\lambda=(6,4,3,3,3,1)$ and $r=8$:
 \[
\begin{array}{c}
 \young(12568\thirteen,379\sixteen,4\twelve\seventeen,\ten\fifteen\nineteen,\eleven\eighteen\twenty,\fourteen)
\end{array}
\leftrightarrow\left(\begin{array}{c} 
\young(12568,37,4)\end{array}
,\begin{array}{c}
  \young(:::::5,::18,:49,27\eleven,3\ten\twelve,6)
 \end{array}
\right)\]
\end{example}

\begin{lemma}
 Let $\lambda\vdash n,\sigma\in S_r$ with $r\leq n$. For two standard Young tableaux $T,T'\in SYT(\lambda)$, set $\phi_r(T)=(U,V)$ and $\phi_r(T')=(U',V')$, then 
 \[\pi^{\lambda}(\sigma)_{T,T'}=\left\{\begin{array}{lcr}0&\mbox{if}&V\neq V'\\      \pi^{\nu}(\sigma)_{U,U'}          &\mbox{if}&V=V',     \end{array}\right.\]
 where $\nu=\sh(U)=\sh(U')$ is the shape of the tableau $U$ in the second case.
 \end{lemma}
\begin{proof}
We will prove the lemma by induction on the number of factors in the (minimal) decomposition of $\sigma$ into adjacent transpositions. Recall that if $\sigma=(k,k+1)$, $k<r$, by Definition \ref{def young seminormal} 
\[\pi^{\lambda}((k,k+1))_{T,T'}=\left\{\begin{array}{cr}
1/d_k(T)&\mbox{ if } T=T';\\
\sqrt{1-\frac{1}{d_k(T)^2}}&\mbox{ if } (k,k+1)T=T'.\\
0&\mbox{else.}\end{array}\right.
\]
Hence if $V\neq V'$ then $\pi^{\lambda}((k,k+1))_{T,T'}=0$, otherwise $\pi^{\lambda}((k,k+1))_{T,T'}=\pi^{\nu}((k,k+1))_{U,U'}$ since $d_k(T)=d_k(U)$.
\smallskip

Consider a general $\sigma\in S_r$ and write it as $\sigma=(k,k+1)\tilde{\sigma}$. Then
\[\pi^{\lambda}(\sigma)_{T,T'}=\sum_{S\in \SYT(\lambda)}\pi^{\lambda}((k,k+1))_{T,S}\pi^{\lambda}(\tilde{\sigma})_{S,T'}\]
\[\hspace{-0.6cm}=\frac{1}{d_k(T)}\pi^{\lambda}(\tilde{\sigma})_{T,T'}+\sqrt{1-\frac{1}{d_k(T)^2}}\pi^{\lambda}(\tilde{\sigma})_{(k,k+1)T,T'}\delta_{\{(k,k+1)T\in \SYT(\lambda)\}}.\]
We apply the inductive hypothesis on $\tilde{\sigma}$, obtaining 
\begin{itemize}
 \item $\frac{1}{d_k(T)}\pi^{\lambda}(\tilde{\sigma})_{T,T'}=\left\{\begin{array}{lcr}0&\mbox{if}&V\neq V'\\      \frac{1}{d_k(U)}\pi^{\nu}(\tilde{\sigma})_{U,U'} &\mbox{if}&V=V',     \end{array}\right.$
 \item $\sqrt{1-\frac{1}{d_k(T)^2}}\pi^{\lambda}(\tilde{\sigma})_{(k,k+1)T,T'}=\left\{\begin{array}{lcr}0&\mbox{if}&V\neq V'\\ \sqrt{1-\frac{1}{d_k(U)^2}}\pi^{\lambda}(\tilde{\sigma})_{(k,k+1)U,U'}&\mbox{if}&V=V',     \end{array}\right.$
\end{itemize}
since if $(k,k+1)T$ is a standard Young tableau, then $\phi_r((k,k+1)T)=((k,k+1)U,V)$. Similarly, $\delta_{\{(k,k+1)T\in \SYT(\lambda)\}}=\delta_{\{(k,k+1)U\in \SYT(\nu)\}}$. To conclude
\[\frac{1}{d_k(U)}\pi^{\lambda}(\tilde{\sigma})_{U,U'}+\sqrt{1-\frac{1}{d_k(U)^2}}\pi^{\lambda}(\tilde{\sigma})_{(k,k+1)U,U'}\delta_{\{(k,k+1)U\in \SYT(\nu)\}}=\pi^{\lambda}(\sigma)_{U,U'}\qedhere\]
\end{proof}

\begin{corollary}\label{corollary on total sum}
 Let $\lambda$ and $\sigma$ be as before, then 
 \[TS^{\lambda}(\sigma)=\sum_{\nu\vdash r}TS^{\nu}(\sigma)\cdot\dim\nu\cdot\frac{\dim\lambda/\nu}{\dim\lambda}.\]
\end{corollary}
\begin{proof}
 By the previous lemma:
 \[TS^{\lambda}(\sigma):=\sum_{T,T'\in \SYT(\lambda)} \frac{\pi^{\lambda}(\sigma)_{i,j}}{\dim\lambda}=\sum_{\nu\vdash r}\sum_{U,U'\in\SYT(\nu)}\pi^{\nu}(\sigma)_{U,U'}\cdot\frac{\dim\lambda/\nu}{\dim\lambda},\]
 and the conclusion is immediate.
\end{proof}

\begin{lemma}\label{lemma on skew dimension}
 Let $|\nu|=r\leq n=|\lambda|$, then
 \[\dim\lambda/\nu=\frac{1}{r!}\sum_{\tau\in S_r}\chi^{\nu}(\tau)\chi^{\lambda}(\tau).\]
\end{lemma}
\begin{proof}
Fix $\nu$, we prove the statement by induction on $n$. If $n=r$ then the LHS is equal to $\delta_{\lambda,\nu}$, the Kronecker delta, and same for the RHS by the character orthogonality relation of the first kind.
 \smallskip
 
 Suppose $r=n-1$, then by \eqref{decomposition of the representation matrix} we have $\chi^{\lambda}(\tau)=\sum_{\mu_j\nearrow\lambda}\chi^{\mu_j}(\tau)$ for each $\tau\in S_{n-1}$. Hence the RHS of the statement can be written as 
 \[\frac{1}{r!}\sum_{\tau\in S_r}\chi^{\nu}(\tau)\chi^{\lambda}(\tau)=\sum_{\mu_j\nearrow\lambda}\left(\frac{1}{r!}\sum_{\tau\in S_r}\chi^{\nu}(\tau)\chi^{\mu_j}(\tau)\right)=\sum_{\mu_j\nearrow\lambda}\delta_{\mu_j=\nu}=\delta_{\nu\nearrow\lambda},\]
 which is equal to the LHS.
 \smallskip
 
 By the inductive hypothesis we consider the statement true for each $\tilde{\lambda}\vdash n-1$. Hence 
 \begin{align*}
  \dim\lambda/\nu&=\sum_{\tilde{\lambda}\vdash n-1}\dim\lambda/\tilde{\lambda}\cdot\dim\tilde{\lambda}/\nu\\
  &=\sum_{\tilde{\lambda}\nearrow\lambda}\left(\frac{1}{r!}\sum_{\tau\in S_r}\chi^{\nu}(\tau)\chi^{\tilde{\lambda}}(\tau)\right)\\
  &=\frac{1}{r!}\sum_{\tau\in S_r}\chi^{\nu}(\tau)\chi^{\lambda}(\tau).\qedhere
 \end{align*} 
\end{proof}

\begin{proposition}\label{formula con traccia}
 Let $\sigma\in S_r$, $\lambda\vdash n$, then
 \[TS^{\lambda}(\sigma)=\sum_{\tau\in S_r}\mathbb{E}_{PL}^{r}\left[\hat{\chi}^{\nu}(\tau) TS^{\nu}(\sigma)\right]\hat{\chi}^{\lambda}(\tau).\]
\end{proposition}

\begin{proof}
We apply Corollary \ref{corollary on total sum} and Lemma \ref{lemma on skew dimension}:
 \begin{align*}
 TS^{\lambda}(\sigma)&=\sum_{\nu\vdash r}TS^{\nu}(\sigma)\cdot\dim\nu\left(\frac{1}{r!}\sum_{\tau\in S_r}\chi^{\nu}(\tau)\hat{\chi}^{\lambda}(\tau)\right)\\
&=\sum_{\tau\in S_r}\left(\sum_{\nu\vdash r}\frac{\dim\nu^2}{r!}\hat{\chi}^{\nu}(\tau)TS^{\nu}(\sigma)\right)\hat{\chi}^{\lambda}(\tau)\\
\end{align*}
and we recognize inside the parenthesis the average $\mathbb{E}_{PL}^{r}\left[\hat{\chi}^{\nu}(\tau) TS^{\nu}(\sigma)\right]$ taken with the Plancherel measure of partitions of $r$.
\end{proof}

\begin{proof}[Proof of Theorem \ref{convergence of total sum}]
Consider a permutation $\sigma$, we write the previous proposition as 
\[TS^{\lambda}(\sigma)=m+v\cdot\hat{\chi}^{\lambda}((1,2))+\sum_{\substack{\tau\in S_r\\\wt(\tau)>2}}\mathbb{E}_{PL}^{r}\left[\hat{\chi}^{\nu}(\tau) TS^{\nu}(\sigma)\right]\hat{\chi}^{\lambda}(\tau),\]
where the first two terms correspond respectively to $\tau=\id$ and the sum of all transpositions. By Kerov's Theorem \ref{convergence of characters}, we have that 
\[\sqrt{n}\sum_{\substack{\tau\in S_r\\\wt(\tau)>2}}\mathbb{E}_{PL}^{r}\left[\hat{\chi}^{\nu}(\tau) TS^{\nu}(\sigma)\right]\hat{\chi}^{\lambda}(\tau)\overset{p}\to 0\]
and $\sqrt{n} \cdot v\cdot \hat{\chi}^{\lambda}((1,2))\overset{d}\to v\mathcal{N}(0,2)$. Consider now a sequence of permutations $\sigma_1,\sigma_2,\ldots$, then
 \[\left\{\sqrt{n}\cdot(TS^{\lambda}(\sigma_i)-m_i)\right\}_{i\geq1}=\left\{n\cdot v_i\hat{\chi}^{\lambda}((1,2))+o_P(1)\right\}_{i\geq1}\overset{d}\to\left\{ \mathcal{N}(0,2 v_i^2)\right\}_{i\geq1}.\qedhere\]
\end{proof}

\begin{example}
 Let $\sigma=(r-1,r)$, $r>2$, be an adjacent transposition. In this example we will show that $m=\mathbb{E}_{PL}^{r}\left[TS^{\nu}(\sigma)\right]$ is strictly positive and that $v:=\binom{r}{2}\mathbb{E}_{PL}^{r}\left[\hat{\chi}^{\nu}_{(2,1,\ldots,1)} TS^{\nu}(\sigma)\right]=1$. First, notice that, for all $\nu\vdash r$ and for all $T,S$ standard Young tableaux of shape $\nu$, $d_{r-1}(T')=-d_{r-1}(T)$, where $T'$ is a tableau of shape $\nu'$ conjugated to $T$, which implies that
 \[\pi^{\nu}(\sigma)_{T',S'}=\left\{\begin{array}{lcr}
                                           -\pi^{\nu}(\sigma)_{T,S}&\mbox{if}& T=S\\
                                           \pi^{\nu}(\sigma)_{T,S}&\mbox{if}& T\neq S,
                                           \end{array}\right.\]
        and if $T\neq S$ then $ \pi^{\nu}(\sigma)_{T,S}\geq 0$. Hence
\[  m=\sum_{\nu\vdash r}\frac{(\dim\nu)^2}{r!}TS^{\nu}(\sigma)=\frac{1}{2}\sum_{\nu\vdash r}\frac{(\dim\nu)^2}{r!}(TS^{\nu}(\sigma)+TS^{\nu'}(\sigma)).\]
Since $TS^{\lambda}(\sigma)=\hat{\chi}^{\lambda}(\sigma)+\sum_{\substack{T,S\in \SYT(\nu)\\T\neq S}}\frac{\pi^{\nu}(\sigma)_{T,S}}{\dim\nu}$, then
\[m=\frac{1}{2}\sum_{\nu\vdash r}\frac{\dim\nu}{r!}\left(\sum_{\substack{T,S\in \SYT(\nu)\\T\neq S}}\pi^{\nu}(\sigma)_{T,S}+\pi^{\nu'}(\sigma)_{T',S'}\right)>0.\]
For $r>2$, at least one summand is nonzero.

Consider now $v$ and decompose $TS^{\nu}(\sigma)$ as above:
 \[v=\binom{r}{2}\left(\mathbb{E}_{PL}^{r}\left[\hat{\chi}^{\nu}_{(2,1,\ldots,1)} \hat{\chi}^{\nu}(\sigma)\right]+\mathbb{E}_{PL}^{r}\left[\hat{\chi}^{\nu}_{(2,1,\ldots,1)}\sum_{\substack{T,S\in \SYT(\nu)\\T\neq S}}\frac{\pi^{\nu}(\sigma)_{T,S}}{\dim\nu}\right]\right).\]
 By the character relations of the second kind we get 
 \[\mathbb{E}_{PL}^{r}\left[\hat{\chi}^{\nu}_{(2,1,\ldots,1)} \hat{\chi}^{\nu}(\sigma)\right]=\sum_{\nu\vdash r}\frac{(\dim\nu)^2}{r!}\hat{\chi}^{\nu}_{(2,1,\ldots,1)}\hat{\chi}^{\nu}(\sigma)=\frac{2}{r(r-1)}.\]
 On the other hand, using $T'$ for the conjugate of $T$ as above, 
 \[\mathbb{E}_{PL}^{r}\left[\hat{\chi}^{\nu}_{(2,1,\ldots,1)}\sum_{\substack{T,S\in \SYT(\nu)\\T\neq S}}\frac{\pi^{\nu}(\sigma)_{T,S}}{\dim\nu}\right]=\]
 \[\frac{1}{2}\sum_{\nu\vdash r}\frac{1}{r!}\sum_{\substack{T,S\in \SYT(\nu)\\T\neq S}}\left(\chi^{\nu}_{(2,1,\ldots,1)}\pi^{\nu}(\sigma)_{T,S}+\chi^{\nu'}_{(2,1,\ldots,1)}\pi^{\nu'}(\sigma)_{T',S'}\right)=0\]
 since $\pi^{\nu}(\sigma)_{T,S}=\pi^{\nu'}(\sigma)_{T',S'}$ if $T\neq S$ and $\chi^{\nu}_{(2,1,\ldots,1)}=-\chi^{\nu'}_{(2,1,\ldots,1)}$. Finally, $v=1$.
\end{example}
Here we write the values of $m= \mathbb{E}_{PL}^{r}\left[TS^{\nu}(\sigma)\right]$ and $v=\binom{r}{2}\mathbb{E}_{PL}^{r}\left[\hat{\chi}^{\nu}_{(2,1,\ldots,1)} TS^{\nu}(\sigma)\right]$ when $\sigma\in S_4$.
\begin{center}
 \scalebox{0.8}{
\begin{tabular}{ |c|ccccccccccccc|} 
 \hline
$\sigma$&$\id$&(3,4)&(2,3)&(2,3,4)&(2,4,3)&(2,4)&(1,2)&(1,2)(3,4)&(1,2,3)&(1,2,3,4)&(1,2,4,3)&(1,2,4)&(1,3,2)\\
$m$&1&1/2&2/3&5/12&1/6&-1/4&0&0&1/3&1/3&1/3&0&-1/3\\
$v$&0&1&1&1/2&4/3&13/6&1&1&0&0&2/3&1/3&0\\
 \hline
\end{tabular}}\smallskip

 \scalebox{0.8}{\begin{tabular}{|c|ccccccccccc|}
\hline
$\sigma$&(1,3,4,2)&(1,3)&(1,3,4)&(1,3)(2,4)&(1,3,2,4)&(1,4,3,2)&(1,4,2)&(1,4,3)&(1,4)&(1,4,2,3)&(1,4)(2,3)\\
$m$&-1/12&-2/3&-1/6&7/12&1/6&-1/12&0&-5/12&-1/4&-2/3&-7/12\\
$v$&1/2&1&0&-7/6&-1/3&-7/6&-4/3&-5/6&-1/6&1/3&1/6\\
\hline
\end{tabular}}
\end{center}

\subsection{Partial sum of the entries of an irreducible representation}
\begin{defi}
Let $\lambda\vdash n$, $\sigma\in S_n, u\in \R$. Then
\[PS_u^{\lambda}(\sigma):=\sum_{i,j\leq u \dim\lambda}\frac{\pi^{\lambda}(\sigma)_{i,j}}{\dim\lambda}.\]
 \end{defi}
 We can now argue in a similar way as we did in Section \ref{chapter the partial trace}: summing entries of $\pi^{\lambda}(\sigma)$ up to a certain index is equivalent to sum all the entries of the submatrices $\pi^{\mu_j}(\sigma)$ for $j<\bar{j}$ and the right choice of $\bar{j}$, plus a remainder which is again a partial sum. We present the analogous of Proposition \ref{first order decomposition of the partial trace}; we omit the proof, since it follows the same argument of the partial trace version:
\begin{proposition}\label{first order decomposition of partial sum}
 Fix $u\in [0,1]$, $\lambda\vdash n$ and $\sigma\in S_r$ with $r\leq n-1$. Set $\left(F_{ct}^{\lambda}\right)^{*}(u)=\frac{y_{\bar{j}}}{\sqrt{n}}=v^{\lambda}$. Define
 \[\bar{u}=\frac{\dim\lambda}{\dim\mu_{\bar{j}}}\left(u-\sum_{j<\bar{j}}\dim\mu_j\right)<1, \]
 then
\begin{equation}\label{first decomposition of partial sum}
 PS_u^{\lambda}(\sigma)=\sum_{y_j<v^{\lambda}\sqrt{n}}\frac{\dim\mu_j}{\dim\lambda}TS^{\mu_j}(\sigma)+\frac{\dim\mu_{\bar{j}}}{\dim\lambda}PS_{\bar{u}}^{\mu_{\bar{j}}}(\sigma).
\end{equation}
\end{proposition}
As before, we call the \emph{main term for the partial sum} $MS_u^{\lambda}(\sigma):=\sum_{y_j<v^{\lambda}\sqrt{n}}\frac{\dim\mu_j}{\dim\lambda}TS^{\mu_j}(\sigma)$ and \emph{remainder for the partial sum} $RS_u^{\lambda}(\sigma):=\frac{\dim\mu_{\bar{j}}}{\dim\lambda}PS_{\bar{u}}^{\mu_{\bar{j}}}(\sigma)$.

The connection between the main term for the partial trace and the main term for the partial sum is easily described by applying Proposition \ref{formula con traccia}:
\begin{align*}
 MS_u^{\lambda}(\sigma)&=\sum_{y_j<v^{\lambda}\sqrt{n}}\frac{\dim\mu_j}{\dim\lambda}TS^{\mu_j}(\sigma)\\
 &=\sum_{ y_j<v^{\lambda}\sqrt{n}}\frac{\dim\mu_j}{\dim\lambda}\sum_{\tau\in S_r}\mathbb{E}_{PL}^{r}\left[\hat{\chi}^{\nu}(\tau) TS^{\nu}(\sigma)\right]\hat{\chi}^{\mu_j}(\tau)\\
 &=\sum_{\tau\in S_r}\mathbb{E}_{PL}^{r}\left[\hat{\chi}^{\nu}(\tau) TS^{\nu}(\sigma)\right]\left(  \sum_{ y_j<v^{\lambda}\sqrt{n}}\frac{\dim\mu_j}{\dim\lambda}\hat{\chi}^{\mu_j}(\tau)\right)\\
 &=\sum_{\tau\in S_r}\mathbb{E}_{PL}^{r}\left[\hat{\chi}^{\nu}(\tau) TS^{\nu}(\sigma)\right]MT_u^{\lambda}(\tau),\numberthis\label{formula between G and F}
\end{align*}
We can thus apply Theorem \ref{convergence of transformed co-transition} on the convergence of the main term for the partial trace to describe the asymptotic of $MS_u^{\lambda}(\sigma)$; notice that the only term in the previous sum which does not disappear when $\lambda$ increases is the one in which $\tau$ is the identity; moreover, the second highest degree term correspond to $\tau$ being a transposition, and thus of order $O_P(n^{-\frac{1}{2}})$. We get:
\begin{corollary}\label{convergence of G}
Set $\sigma_1\in S_{r_1},\sigma_2\in S_{r_2},\ldots$, $u_1,u_2,\ldots\in[0,1]$ and let $\lambda\vdash n$. Consider as before
\[m_i:= \mathbb{E}_{PL}^{r_i}\left[TS^{\nu}(\sigma_i)\right]\qquad \mbox{and}\qquad v_i:=\binom{r_i}{2}\mathbb{E}_{PL}^{r_i}\left[\hat{\chi}^{\nu}_{(2,1,\ldots,1)} TS^{\nu}(\sigma_i)\right].\]
Then 
 \[\left\{\sqrt{n}\cdot\left(MS_{u_i}^{\lambda}(\sigma_i)-u_i\cdot m_i\right)\right\}\overset{d}\to \left\{u_i\cdot\mathcal{N}(0,2 v_i^2)\right\}.\]
\end{corollary}
\begin{proof}
 For a generic $u$ and $\sigma$, we rewrite \eqref{formula between G and F} as:
 \[MS_u^{\lambda}(\sigma)=m\cdot MT_u^{\lambda}(\id) + \sum_{\substack{\tau\in S_r\\\wt(\tau)=2}}\mathbb{E}_{PL}^{r}\left[\hat{\chi}^{\nu}(\tau) TS^{\nu}(\sigma)\right]MT_u^{\lambda}(\tau)+ \sum_{\substack{\tau\in S_r\\\wt(\tau)>2}}\mathbb{E}_{PL}^{r}\left[\hat{\chi}^{\nu}(\tau) TS^{\nu}(\sigma)\right]MT_u^{\lambda}(\tau).\]
 By Lemma \ref{behaviour of co-transition}, $MT_u^{\lambda}(\id)\overset{p}\to u$. By Theorem \ref{convergence of transformed co-transition}
 \[n^{\frac{1}{2}}\sum_{\substack{\tau\in S_r\\\wt(\tau)=2}}\mathbb{E}_{PL}^{r}\left[\hat{\chi}^{\nu}(\tau) TS^{\nu}(\sigma)\right]MT_u^{\lambda}(\tau)\overset{d}\to v\cdot u\cdot \mathcal{N}(0,2);\]
 on the other hand 
 \[\sum_{\substack{\tau\in S_r\\\wt(\tau)>2}}\mathbb{E}_{PL}^{r}\left[\hat{\chi}^{\nu}(\tau) TS^{\nu}(\sigma)\right]MT_u^{\lambda}(\tau)\in o_P(n^{-\frac{1}{2}}).\]
Thus 
 \[\left\{\sqrt{n}\cdot\left(MS_{u_i}^{\lambda}(\sigma_i)-u_i\cdot m_i\right)\right\}\overset{d}\to \left\{u_i\cdot\mathcal{N}(0,2 v_i^2)\right\}.\qedhere\]
\end{proof}

Although we cannot show a satisfying result on the convergence of the partial trace because we cannot prove that $n^{\frac{wt(\rho)}{2}}\frac{\dim\mu_{\bar{j}}}{\dim\lambda}PT_{\bar{u}}^{\mu_{\bar{j}}}(\sigma)\to 0$, we are more lucky with the partial sum:

\begin{theorem}{\label{convergence of partial sum}}
 Set $\sigma\in S_r$ and let $\lambda\vdash n$. Set $m:= \mathbb{E}_{PL}^{r}\left[TS^{\nu}(\sigma)\right]$
Then 
 \[ PS_u^{\lambda}(\sigma)\overset{p}\to u\cdot m.\]
\end{theorem}

We will prove the theorem after three lemmas. 
\begin{lemma}\label{bound on nonzero terms}
 Let as usual $\sigma\in S_r$ and $\lambda\vdash n$, with $n>r$. Then $\pi^{\lambda}(\sigma)_{i,j}=0$ for all $i,j$ such that $|i-j|>r!$
\end{lemma}
\begin{proof}
We iteratively use the Formula \eqref{decomposition of the representation matrix}:
\[\pi^{\lambda}(\sigma)=\bigoplus_{\mu^{(n-1)}\nearrow \lambda}\pi^{\mu^{(n-1)}}(\sigma)=\bigoplus_{\substack{\mu^{(n-1)}\nearrow \lambda\\ \mu^{(n-2)}\nearrow \mu^{(n-1)}}}\pi^{\mu^{(n-2)}}(\sigma)=
\bigoplus_{\mu^{(r)}\nearrow \cdots\nearrow\lambda}\pi^{\mu^{(r)}}(\sigma).\]
Therefore $\pi^{\lambda}(\sigma)$ is a block matrix such that the only nonzero blocks are those on the diagonal. To conclude, notice that $\dim{\mu^{(r)}}\leq r!$
\end{proof}
\begin{lemma}\label{bound on entries}
 With the usual setting of $\lambda$, $\sigma$ and $u$ we have
 \[|PS_u^{\lambda}(\sigma)|\leq 2 u\cdot r!\cdot 2^{l(\sigma)},\]
 where $l(\sigma)$ is the length of the reduced word of $\sigma$, \emph{i.e.} the minimal number of adjacent transpositions occurring in the decomposition of $\sigma$.
\end{lemma}
\begin{proof}
 We first prove a bound on the absolute value of an entry of the matrix
 \begin{equation}\label{bound of an entry}
  |\pi^{\lambda}(\sigma)_{T,S}|\leq 2^{l(\sigma)}\mbox{ for each }T,S
 \end{equation}
 by induction on $l(\sigma)$. The initial case is $\sigma=(k,k+1)$, for which $|\pi^{\lambda}((k,k+1))_{T,S}|\leq 1$.\\
 Suppose $\sigma=(k,k+1)\tilde{\sigma}$, then
 \[|\pi^{\lambda}(\sigma)_{T,S}|=\left|\frac{1}{d_k(T)}\pi^{\lambda}(\tilde{\sigma})_{T,S}+\sqrt{1-\frac{1}{d_k(T)^2}}\pi^{\lambda}(\tilde{\sigma})_{(k,k+1)T,S}\delta_{(k,k+1)T\in SYT(\lambda)}\right|\leq 2^{l(\tilde{\sigma})}+2^{l(\tilde{\sigma})}=2^{l(\sigma)}.\]
  We see that
\[ |PS_u^{\lambda}(\sigma)|=\left|\sum_{i,j\leq u \dim\lambda}\frac{\pi^{\lambda}(\sigma)_{i,j}}{\dim\lambda}\right|
 \leq 2u\dim\lambda\cdot r!\max_{T,S\in \SYT(\lambda)}\left|\frac{\pi^{\lambda}(\sigma)_{T,S}}{\dim\lambda}\right|\le2^{l(\sigma)+1} u\cdot r!, \]
which allows us to conclude. Notice that we used the previous lemma in the first inequality, which shows that the number of nonzero terms appearing in the sum is bounded by $2 u\dim\lambda\cdot r!$
\end{proof}
\begin{lemma}
 Let $\sigma\neq \id$, then $PT_u^{\lambda}(\sigma)\overset{p}\to 0.$
\end{lemma}
\begin{proof}
 Recall the decomposition of the partial trace into main term and remainder (Proposition \ref{first order decomposition of the partial trace}):
 \[PT_u^{\lambda}(\sigma)=MT_u^{\lambda}(\sigma)+RT_u^{\lambda}(\sigma).\]
 By Theorem \ref{convergence of transformed co-transition}, if $\sigma\neq\id$ then $MT_u^{\lambda}(\sigma)\in o_P(n^{-\frac{\wt(\sigma)}{2}})\subseteq o_P(1)$. On the other hand we can estimate the remainder through the bound on a singular entry calculated in \eqref{bound of an entry}:
 \[|RT_u^{\lambda}(\sigma)|\leq 2^{l(\sigma)}\cdot\frac{\dim\mu_{\bar{j}}}{\dim\lambda}\overset{p}\to 0,\]
 where $l(\sigma)$ is the length of the reduced word of $\sigma$, $\mu_{\bar{j}}$ is the subpartition corresponding to the renormalized content $\frac{y_{\bar{j}}}{\sqrt{n}}=(F_{ct}^{\lambda})^*(u)$, and 
 \[\bar{u}=\frac{\dim\lambda}{\dim\mu_{\bar{j}}}\left(u-\sum_{j<\bar{j}}\dim\mu_j\right).\qedhere \]
\end{proof}

\begin{proof}[Proof of Proposition \ref{convergence of partial sum}]
We claim that the partial sum $PS_u^{\lambda}(\sigma)$ and the following quantity are asymptotically close 
\[\sum_{\tau\in S_r}\mathbb{E}_{PL}^{r}\left[\hat{\chi}^{\nu}(\tau) TS^{\nu}(\sigma)\right]PT_u^{\lambda}(\tau),\]
that is, 
\[\left|PS_u^{\lambda}(\sigma)-\sum_{\tau\in S_r}\mathbb{E}_{PL}^{r}\left[\hat{\chi}^{\nu}(\tau) TS^{\nu}(\sigma)\right]PT_u^{\lambda}(\tau)\right|\overset{p}\to 0.\]
We substitute in the previous expression the decomposition formulas for the partial sum and partial trace, respectively Propositions \ref{first order decomposition of partial sum} and \ref{first order decomposition of the partial trace}, and we simplify according to the equality of Equation (\ref{formula between G and F}), so that it remains:
\[\frac{\dim\mu_{\bar{j}}}{\dim\lambda}\left|PS_{\bar{u}}^{\mu_{\bar{j}}}(\sigma)-\sum_{\tau\in S_r}\mathbb{E}_{PL}^{r}\left[\hat{\chi}^{\nu}(\tau) TS^{\nu}(\sigma)\right]PT_{\bar{u}}^{\mu_{\bar{j}}}(\tau)\right|.\]
We recall from \eqref{bound of an entry} that $|\pi^{\lambda}(\tau)_{T,S}|\leq 2^{l(\tau)}$ for each $T,S$, which implies that 
\[\left|PT_{\bar{u}}^{\mu_{\bar{j}}}(\tau)\right|\leq 2^{l(\tau)}\cdot \bar{u}\]
hence
\begin{align*}
 &\frac{\dim\mu_{\bar{j}}}{\dim\lambda}\left|PS_{\bar{u}}^{\mu_{\bar{j}}}(\sigma)-\sum_{\tau\in S_r}\mathbb{E}_{PL}^{r}\left[\hat{\chi}^{\nu}(\tau) TS^{\nu}(\sigma)\right]PT_{\bar{u}}^{\mu_{\bar{j}}}(\tau)\right|\\
&\leq \frac{\dim\mu_{\bar{j}}}{\dim\lambda} \left(2 \bar{u}\cdot r!\cdot 2^{l(\sigma)}+\sum_{\tau\in S_r}\mathbb{E}_{PL}^{r}\left[\hat{\chi}^{\nu}(\tau) TS^{\nu}(\sigma)\right]2^{l(\tau)}\cdot \bar{u}\right)\overset{p}\to 0
\end{align*}
since the expression inside the parenthesis is bounded and $\dim\mu_{\bar{j}}/\dim\lambda\overset{p}\to 0$.\\
On the other hand 
\[\sum_{\tau\in S_r}\mathbb{E}_{PL}^{r}\left[\hat{\chi}^{\nu}(\tau) TS^{\nu}(\sigma)\right]PT_u^{\lambda}(\tau)=\frac{\lfloor u\dim\lambda\rfloor}{\dim\lambda}\mathbb{E}_{PL}^{r}\left[TS^{\nu}(\tau)\right]+\smallo(1)\overset{d}\to u\mathbb{E}_{PL}^{r}\left[TS^{\nu}(\tau)\right].\qedhere\]
\end{proof}

For the same reasons for which we cannot prove convergence of the partial trace, here we cannot show a second order asymptotic, indeed we know that the term $\frac{\dim\mu_{\bar{j}}}{\dim\lambda}PS_{\bar{u}}^{\mu_{\bar{j}}}(\sigma)$ disappear when $\lambda$ grows, but we do not know how fast. This will be discussed more deeply in the next section.

We can now generalize the concept of partial sum, and Lemma \ref{bound on entries} allows us to describe its asymptotics.
\begin{defi}
 Let $\lambda\vdash n$, $\sigma\in S_n, u_1,u_2\in \R$. Define the {\it partial sum of the entries of the irreducible representation matrix associated to $\pi^{\lambda}(\sigma)$ stopped at $(u_1,u_2)$} as
\[PS_{u_1,u_2}^{\lambda}(\sigma):=\sum_{\substack{i\leq u_1 \dim\lambda\\j\leq u_2 \dim\lambda}}\frac{\pi^{\lambda}(\sigma)_{i,j}}{\dim\lambda}.\]
\end{defi}
\begin{corollary}
Let $u_1,u_2\in [0,1]$, then
\[ PS_{u_1,u_2}^{\lambda}(\sigma)\overset{p}\to \min\{u_1,u_2\}\mathbb{E}_{PL}^{r}\left[TS^{\nu}(\sigma)\right]\]
\end{corollary}
\begin{proof}
 Suppose $u_1<u_2$, we only need to prove that 
 \[|PS_{u_1,u_2}^{\lambda}(\sigma)-PS_{u_1}^{\lambda}(\sigma)|=\left|\sum_{\substack{i\leq u_1 \dim\lambda\\u_1 \dim\lambda<j\leq u_2 \dim\lambda}}\frac{\pi^{\lambda}(\sigma)_{i,j}}{\dim\lambda}\right|\overset{p}\to 0.\]
 The number of nonzero terms in the above sum is bounded by $(r!)^2$ by Lemma \ref{bound on nonzero terms}; moreover by \hbox{Equation \eqref{bound of an entry}} we have \hbox{$|\pi^{\lambda}(\sigma)_{i,j}|\leq 2^{l(\sigma)}$}. Therefore
  \[|PS_{u_1,u_2}^{\lambda}(\sigma)-PS_{u_1}^{\lambda}(\sigma)|\overset{p}\to 0,\]
 and the corollary follows.
\end{proof}

\section{A conjecture}\label{section conjecture}

\begin{conjecture}\label{conjecture 1}
 Set as usual $\lambda\vdash n$. We say that $\mu\subseteq \lambda$ if $\mu$ is a partition obtained from $\lambda$ by removing boxes.
 \smallskip
 
 We conjecture that there exists $\alpha>0$ such that, for all $s$,
\[P_{PL}^n\left(\left\{\lambda: \max_{\substack{\mu\colon \mu\subseteq \lambda\\|\mu|=|\lambda|-s}}\frac{\dim\mu}{\dim\lambda}>n^{-\alpha s}\right\}\right)\to 0.\]
 \end{conjecture}
 We run some tests which hint the conjecture to be true for $s=1$; for example, if $\alpha=0.2$ and $n\leq 70$, we have
 \[P_{PL}^n\left(\left\{\lambda: \max_{\substack{\mu\colon \mu\subseteq \lambda\\|\mu|=|\lambda|-s}}\frac{\dim\mu}{\dim\lambda}>n^{-\alpha s}\right\}\right)\leq \left\{\begin{array}{cc}
                                                                                                                                                                         0.2& \mbox{ if } n\geq 7\\
                                                                                                                                                                         0.1&\mbox{ if } n\geq 12\\
                                                                                                                                                                          0.05&\mbox{ if } n\geq 37.\\
                                                                                                                                                                        \end{array}\right.\]
 
Notice that, for $s>1$,
 \[\max_{\substack{\mu\colon \mu\subseteq \lambda\\|\mu|=|\lambda|-s}}\frac{\dim\mu}{\dim\lambda}\leq \max_{\mu^{(1)}\nearrow\mu^{(2)}\nearrow\ldots\nearrow \mu^{(s-1)}\nearrow\lambda}\prod_i\left(\max_{\substack{\mu\colon \mu\subseteq \mu^{(i)}\\|\mu|=|\mu^{(i)}|-1}}\frac{\dim\mu}{\dim\mu^{(i)}}\right),\]
so that it may seem that it is enough to prove the conjecture just for $s=1$. 

This is not true though, since the sequence $\mu^{(1)}\nearrow\ldots\nearrow\mu^{(s-1)}$ in the RHS above is not Plancherel distributed. Hence we need the Conjecture in the more general form.
 \begin{proposition}
  If Conjecture \ref{conjecture 1} is correct, then  
  \[\left\{PT_{u_i}^{\lambda}(\sigma_i)\right\}\overset{d}\to\left\{ u_i\prod_{k\geq 2} k^{m_k(\rho_i)/2} \mathcal{H}_{m_k(\rho_i)}(\xi_k)\right\}.\]
 \end{proposition}
\begin{proof}
  Recall from Proposition \ref{decomposition of the partial trace} that for any $s<n-r$, there exists a sequence of partitions \hbox{$\mu^{(0)}\nearrow\mu^{(1)}\nearrow\ldots\nearrow \mu^{(s)}=\lambda$} and a sequence of real numbers $0\leq c_{0},\ldots,c_s<1$ such that
 \[PT_u^{\lambda}(\sigma)=\sum_{i=1}^{s}\frac{\dim\mu^{(i)}}{\dim\lambda}MT_{u^{(i)}}^{\mu^{(i)}}(\sigma)+\frac{\dim\mu^{(0)}}{\dim\lambda}PT_{u^{(0)}}^{\mu^{(0)}}(\sigma).\]
 We consider 
  \[n^{\frac{\wt(\sigma)}{2}}PT_u^{\lambda}(\sigma)=n^{\frac{\wt(\sigma)}{2}}\sum_{i=1}^{s}\frac{\dim\mu^{(i)}}{\dim\lambda}MT_{u^{(i)}}^{\mu^{(i)}}(\sigma)+n^{\frac{\wt(\sigma)}{2}}\frac{\dim\mu^{(0)}}{\dim\lambda}PT_{u^{(0)}}^{\mu^{(0)}}(\sigma).\]
In the first sum of the right hand side the term corresponding to $i=s$ is 
\[n^{\frac{\wt(\sigma)}{2}}MT_{u}^{\lambda}(\sigma)\overset{d}\to u \prod_{k\geq 2} k^{m_k(\rho)/2} \mathcal{H}_{m_k(\rho)}(\xi_k),\]
due to Theorem \ref{convergence of transformed co-transition}. On the other hand the other terms in the first sum of the RHS are of the form $n^{\frac{\wt(\sigma)}{2}}\frac{\dim\mu^{(i)}}{\dim\lambda}MT_{u^{(i)}}^{\mu^{(i)}}(\sigma)$ for $i=1,\ldots,s-1$, and it is easy to see that $(n-s+i)^{\frac{\wt(\sigma)}{2}}MT_{u^{(i)}}^{\mu^{(i)}}(\sigma)$ converge (because of Theorem \ref{convergence of transformed co-transition}), while $\frac{\dim\mu^{(i)}}{\dim\lambda}\in o_P(1)$ because of the convergence of the normalized co-transition distribution towards an
atom free distribution (see Lemma \ref{convergence of co-transition}).
\smallskip
We study thus the term $n^{\frac{\wt(\sigma)}{2}}\frac{\dim\mu^{(0)}}{\dim\lambda}PT_{u^{(0)}}^{\mu^{(0)}}(\sigma)$, and Conjecture \ref{conjecture 1} implies that
 \[\frac{\dim\mu^{(0)}}{\dim\lambda}\leq \max_{\substack{\mu\colon \mu\subseteq \lambda\\|\mu|=|\lambda|-s}}\frac{\dim\mu}{\dim\lambda}\leq n^{-\alpha s},\]
 with high probability. We choose $s$ such that $\wt(\sigma)/2<\alpha s$. An easy application of Inequality \eqref{bound of an entry} implies that $|PT_{u^{(0)}}^{\mu^{(0)}}(\sigma)|\leq u^{(0)}\cdot 2^{l(\sigma)}$, thus we have
 \[n^{\frac{\wt(\sigma)}{2}}\frac{\dim\mu^{(0)}}{\dim\lambda}PT_{u^{(0)}}^{\mu^{(0)}}(\sigma)\leq n^{\frac{\wt(\sigma)}{2}}n^{-\alpha s} u^{(0)}\cdot 2^{l(\sigma)}\to 0,\]
 which implies
 \[n^{\frac{\wt(\sigma)}{2}}\frac{\dim\mu^{(0)}}{\dim\lambda}PT_{u^{(0)}}^{\mu^{(0)}}(\sigma)\in o_P(1).\]
 Therefore 
  \[\left\{PT_{u_i}^{\lambda}(\sigma_i)\right\}\overset{d}\to\left\{ u_i\prod_{k\geq 2} k^{m_k(\rho_i)/2} \mathcal{H}_{m_k(\rho_i)}(\xi_k)\right\},\]
  which concludes the proof.
 \end{proof}

 \section{Acknowledgments}
The author would like to express his gratitude to Valentin F{\'e}ray for introducing him on the subject, many insightful discussions, and several corrections and suggestions in the development of the paper.

This research was founded by SNSF grant SNF-149461: ``Dual combinatorics of Jack polynomials''.

\bibliographystyle{alpha}
\bibliography{courant}{}

\def\cprime{$'$}
\begin{thebibliography}{DDN03}

\bibitem[Bia03]{Biane2003}
Ph. Biane.
\newblock Characters of symmetric groups and free cumulants.
\newblock In {\em Asymptotic combinatorics with applications to mathematical
  physics (St. Petersburg, 2001)}, volume 1815 of {\em Lecture Notes in Math.},
  pages 185--200. Springer, Berlin, 2003.

\bibitem[BOO00]{borodin2000asymptotics}
A.~Borodin, A.~Okounkov, and G.~Olshanski.
\newblock {A}symptotics of {P}lancherel measures for symmetric groups.
\newblock {\em J. Amer. Math. Soc.}, 13(3):481--515, 2000.

\bibitem[DDN03]{d2003brownian}
A.~D'Aristotile, P.~Diaconis, and C.~Newman.
\newblock Brownian motion and the classical groups.
\newblock {\em Lecture Notes-Monograph Series}, 41:97--116, 2003.

\bibitem[DF{\'S}10]{dolkega2010explicit}
M.~Do{\l}{\k{e}}ga, V.~F{\'e}ray, and P.~{\'S}niady.
\newblock Explicit combinatorial interpretation of {K}erov character
  polynomials as numbers of permutation factorizations.
\newblock {\em Adv. in Math.}, 225(1):81--120, 2010.

\bibitem[F{\'e}r12]{feray2012partial}
V.~F{\'e}ray.
\newblock Partial {J}ucys--{M}urphy elements and star factorizations.
\newblock {\em European J. Combin.}, 33(2):189--198, 2012.

\bibitem[F{\'S}11]{feray2007asymptotics}
V.~F{\'e}ray and P.~{\'S}niady.
\newblock Asymptotics of characters of symmetric groups related to {S}tanley
  character formula.
\newblock {\em Ann. of Math.}, 173(2):887--906, 2011.

\bibitem[GH03]{geronimo2003necessary}
J.~Geronimo and T.~Hill.
\newblock Necessary and sufficient condition that the limit of {S}tieltjes
  transforms is a {S}tieltjes transform.
\newblock {\em J. Approx. Theory}, 121(1):54--60, 2003.

\bibitem[Gre92]{GreeneRationalIdentity}
C.~Greene.
\newblock A rational function identity related to the {M}urnaghan-{N}akayama
  formula for the characters of {$S_n$}.
\newblock {\em J. Algebr. Comb.}, 1(3):235--255, 1992.

\bibitem[Hor98]{hora1998central}
A.~Hora.
\newblock Central limit theorem for the adjacency operators on the infinite
  symmetric group.
\newblock {\em Comm. Math. Phys.}, 195(2):405--416, 1998.

\bibitem[IK99]{IvanovKerov1999}
V.~Ivanov and S.~Kerov.
\newblock The algebra of conjugacy classes in symmetric groups, and partial
  permutations.
\newblock {\em Zap. Nauchn. Sem. S.-Peterburg. Otdel. Mat. Inst. Steklov.
  (POMI)}, 256(3):95--120, 1999.

\bibitem[IO02]{ivanov2002kerov}
V.~Ivanov and G.~Olshanski.
\newblock Kerov's central limit theorem for the {P}lancherel measure on {Y}oung
  diagrams.
\newblock In {\em Symmetric functions 2001: surveys of developments and
  perspectives}, pages 93--151. Springer, 2002.

\bibitem[Juc66]{Jucys1966}
A.~Jucys.
\newblock On the {Y}oung operators of the symmetric groups.
\newblock {\em Lithuanian Phys. J.}, {VI}(2):180--189, 1966.

\bibitem[Juc74]{Jucys1974}
A.~Jucys.
\newblock Symmetric polynomials and the center of the symmetric group ring.
\newblock {\em Reports Math. Phys.}, 5:107--112, 1974.

\bibitem[Ker93a]{kerov1993gaussian}
S.~Kerov.
\newblock Gaussian limit for the {P}lancherel measure of the symmetric group.
\newblock {\em C. R. Acad. Sci. Paris}, 316:303--308, 1993.

\bibitem[Ker93b]{kerov1993transition}
S.~Kerov.
\newblock Transition probabilities for continual {Y}oung diagrams and the
  {M}arkov moment problem.
\newblock {\em Funct. Anal. Appl.}, 27(2):104--117, 1993.

\bibitem[KV77]{kerov1977asymptotics}
S.~Kerov and A.~Vershik.
\newblock Asymptotics of the {P}lancherel measure of the symmetric group and
  the limiting form of {Y}oung tableaux.
\newblock In {\em Soviet Math. Dokl}, volume~18, pages 527--531, 1977.

\bibitem[LS77]{logan1977variational}
B.~Logan and L.~Shepp.
\newblock A variational problem for random {Y}oung tableaux.
\newblock {\em Adv. in Math.}, 26(2):206--222, 1977.

\bibitem[Mur81]{Murphy1981}
G.~Murphy.
\newblock A new construction of {Y}oung's seminormal representation of the
  symmetric group.
\newblock {\em J. Algebra}, 69:287--291, 1981.

\bibitem[OO98]{OkOl1998}
A.~Okounkov and G.~Olshanski.
\newblock Shifted {S}chur functions.
\newblock {\em St. Petersburg Math. J.}, 9:239--300, 1998.

\bibitem[Rom15]{romik2015surprising}
D.~Romik.
\newblock {\em The surprising mathematics of longest increasing subsequences},
  volume~4 of {\em Textbooks}.
\newblock Cambridge University Press, 2015.

\bibitem[Sag13]{sagan2013symmetric}
B.~Sagan.
\newblock {\em The symmetric group: representations, combinatorial algorithms,
  and symmetric functions}, volume 203 of {\em Graduate Texts in Mathematics}.
\newblock Springer Science \& Business Media, 2013.

\bibitem[You77]{young1977collected}
A.~Young.
\newblock {\em The collected papers of {A}lfred {Y}oung 1873-1940}.
\newblock Mathematical expositions. University of Toronto Press, 1977.

\end{thebibliography}
\label{sec:biblio}

\end{document}